\documentclass[12pt,twoside]{amsart}
\usepackage{amssymb,amsmath,amsthm, amscd, enumerate, mathrsfs}
\usepackage{graphicx, hhline}
\usepackage[all]{xy}
\usepackage{hyperref}
\title[Adjunction and Inversion of Adjunction]{Adjunction and Inversion of Adjunction}

\address{Osamu Fujino \\ Department of 
Mathematics, Graduate School of Science, 
Kyoto University, Kyoto 606-8502, Japan}
\email{fujino@math.kyoto-u.ac.jp}
\address{Kenta Hashizume \\ Graduate School of Mathematical Sciences, 
The University of Tokyo, 3-8-1 Komaba Meguro-ku Tokyo 153-8914, Japan}
\email{hkenta@ms.u-tokyo.ac.jp}

 
\keywords{adjunction, inversion of adjunction, log 
canonical centers, generalized pairs}

\DeclareMathOperator{\ddiv}{div}
\DeclareMathOperator{\Spec}{Spec}
\DeclareMathOperator{\Supp}{Supp}
\DeclareMathOperator{\Nqklt}{Nqklt}
\DeclareMathOperator{\Nqlc}{Nqlc}
\DeclareMathOperator{\Nlc}{Nlc}
\DeclareMathOperator{\NLC}{NLC}
\DeclareMathOperator{\mult}{mult}
\DeclareMathOperator{\Exc}{Exc}
\DeclareMathOperator{\Gal}{Gal}
\DeclareMathOperator{\rank}{rank}
\DeclareMathOperator{\coeff}{coeff}
\newtheorem{thm}{Theorem}[section]
\newtheorem{lem}[thm]{Lemma}

\newtheorem{conj}[thm]{Conjecture}
\newtheorem{cor}[thm]{Corollary}
\newtheorem*{claim-n}{Claim}

\theoremstyle{definition}
\newtheorem{step}{Step}
\newtheorem{defn}[thm]{Definition}
\newtheorem{rem}[thm]{Remark}
\newtheorem*{ack}{Acknowledgments}  
\makeatletter
    
    \@addtoreset{equation}{section}
\makeatother

\begin{document}

\maketitle

\begin{abstract}
We establish adjunction and inversion 
of adjunction for log canonical centers of arbitrary 
codimension in full generality.  
\end{abstract}

\tableofcontents

\section{Introduction}\label{x-sec1}

Throughout this paper, we will work over $\mathbb C$, the complex 
number filed. 
We establish the following adjunction and inversion of adjunction 
for log canonical centers of arbitrary codimension. 

\begin{thm}\label{x-thm1.1}
Let $X$ be a normal variety and let $\Delta$ be an 
effective $\mathbb R$-divisor on $X$ such that 
$K_X+\Delta$ is $\mathbb R$-Cartier. 
Let $W$ be a log canonical center of $(X, \Delta)$ and 
let $\nu\colon Z\to W$ be the normalization of $W$. 
Then we have the adjunction formula 
$$
\nu^*(K_X+\Delta)=K_Z+B_Z+M_Z
$$ 
with the following properties: 
\begin{itemize}
\item[(A)] $(X, \Delta)$ is log canonical 
in a neighborhood of $W$ if and only if 
$(Z, B_Z+M_Z)$ is an NQC generalized 
log canonical pair, and 
\item[(B)] $(X, \Delta)$ is log canonical 
in a neighborhood of $W$ and 
$W$ is a minimal log canonical center of $(X, \Delta)$ 
if and only if $(Z, B_Z+M_Z)$ is an NQC generalized 
kawamata log terminal pair.  
\end{itemize}
\end{thm}

For the definition of NQC generalized 
log canonical pairs and NQC generalized 
kawamata log terminal pairs, see \cite[Section 2]{han-li}. 
 
For the formulation of adjunction and inversion of adjunction 
for log canonical centers of arbitrary codimension 
in full generality, the notion of b-divisors, which 
was first introduced by Shokurov, is very useful. 
In fact, the $\mathbb{R}$-divisors $B_{Z}$ and $M_{Z}$ in 
Theorem \ref{x-thm1.1} are the traces of certain 
$\mathbb{R}$-b-divisors $\mathbf B$ and $\mathbf M$ 
on $Z$, respectively. 
The precise version of Theorem \ref{x-thm1.1} is: 

\begin{thm}[Adjunction and Inversion of Adjunction]\label{x-thm1.2}
Let $X$ be a normal 
variety and let $\Delta$ be an effective $\mathbb R$-divisor 
on $X$ such that 
$K_X+\Delta$ is $\mathbb R$-Cartier. 
Let $W$ be a log canonical center of $(X, \Delta)$ 
and let $\nu\colon Z\to W$ be the normalization of $W$. 
Then there exist a b-potentially nef $\mathbb R$-b-divisor 
$\mathbf M$ and an $\mathbb R$-b-divisor $\mathbf B$ 
on $Z$ such that 
$\mathbf B_Z$ is effective with 
$$
\nu^*(K_X+\Delta)=\mathbf K_Z+\mathbf M_Z+\mathbf B_Z. 
$$
More precisely, there exists a projective birational morphism 
$p\colon Z'\to Z$ from a smooth quasi-projective variety 
$Z'$ such that 
\begin{itemize}
\item[(i)] $\mathbf M=\overline {\mathbf M_{Z'}}$ and 
$\mathbf M_{Z'}$ is a potentially nef $\mathbb R$-divisor on $Z'$, 
\item[(ii)] $\mathbf K+\mathbf B=\overline {\mathbf K_{Z'}+
\mathbf B_{Z'}}$, 
\item[(iii)] $\Supp \mathbf B_{Z'}$ is a simple 
normal crossing divisor on $Z'$, 
\item[(iv)] $\nu\circ p\left(\mathbf B^{>1}_{Z'}\right)=W\cap \Nlc(X, \Delta)$ 
holds set theoretically, where $\Nlc(X, \Delta)$ denotes the non-lc 
locus of $(X, \Delta)$, and 
\item[(v)] $\nu\circ p\left(\mathbf B^{\geq 1}_{Z'}\right)=W
\cap \left(\Nlc(X, \Delta)\cup \bigcup _{W\not\subset W^\dag}W^\dag
\right)$, where $W^\dag$ runs over log canonical centers of $(X, 
\Delta)$ which do not contain $W$, holds set theoretically.  
\end{itemize}
Hence, $(Z, \mathbf B_Z + \mathbf M_Z)$ is generalized log canonical, 
that is, $\mathbf B^{>1}_{Z'}=0$,  
if and only if $(X, \Delta)$ is log canonical 
in a neighborhood of $W$. 
Moreover, $(Z, \mathbf B_Z + \mathbf M_Z)$ is generalized 
kawamata log terminal, 
that is, $\mathbf B^{\geq 1}_{Z'}=0$, if and only if 
$(X, \Delta)$ is log canonical in a neighborhood of 
$W$ and $W$ is a minimal log canonical center of $(X, \Delta)$. 
We note that $\mathbf M_{Z'}$ is semi-ample when 
$\dim W=1$. We also  
note that if $K_X+\Delta$ is $\mathbb Q$-Cartier then 
$\mathbf B$ and $\mathbf M$ become 
$\mathbb Q$-b-divisors by construction. 
\end{thm} 

In this paper, the 
$\mathbb{R}$-b-divisors $\mathbf B$ 
and $\mathbf M$ in Theorem \ref{x-thm1.2} 
are defined by using the notion of basic 
$\mathbb{R}$-slc-trivial fibrations. 
Here, we explain an alternative definition of $\mathbf B$ 
and $\mathbf M$ for the reader's convenience. 
For the details of Definition \ref{x-def1.3}, 
see \cite[Section 5]{fujino-hashizume} 
and \cite[Definition 2.1]{fujino-hashizume-proc}. 

\begin{defn}[{see \cite[Section 5]{fujino-hashizume}, 
\cite[Definition 2.1]{fujino-hashizume-proc}, and 
Remark \ref{x-rem6.1}}]
\label{x-def1.3} 
Let $(X,\Delta)$, $W$, and $\nu \colon Z\to W$ be 
as in Theorem \ref{x-thm1.2}. 
For any higher birational model $\rho\colon \tilde{Z} \to Z$, 
we consider all prime divisors $T$ over $X$ such 
that $a(T,X,\Delta)=-1$ and the center of $T$ on $X$ is $W$. 
We take a log resolution $f\colon Y\to X$ of 
$(X,\Delta)$ so that $T$ is a prime divisor on $Y$ 
and the induced map $f_{T}\colon T\dashrightarrow \tilde{Z}$ is a morphism. 
We put $\Delta_{T}=(\Delta_{Y}-T)|_{T}$, 
where $\Delta_{Y}$ is defined by $K_{Y}+\Delta_{Y}=f^{*}(K_{X}+\Delta)$. 
For any prime divisor $P$ on $\tilde{Z}$, we define a real number 
$\alpha_{P,T}$ by
$$\alpha_{P,T}
=\sup\{\lambda \in \mathbb R\,|\,(T,\Delta_T+\lambda f^*_TP) 
\ \text{is sub log canonical over the generic point of}\ P\}.$$ 
Then the trace $\mathbf B_{\tilde{Z}}$ of 
$\mathbf B$ on $\tilde{Z}$ is defined by
$$\mathbf B_{\tilde{Z}}=\sum_{P}(1-\underset{T}\inf \alpha_{P,T})P$$
where $P$ runs over prime divisors on $\tilde Z$ and 
$T$ 
runs over prime divisors 
over $X$ such that $a(T,X,\Delta)=-1$ 
and the center of $T$ on $X$ is $W$. 
When $W$ is a prime divisor on $X$, 
$T$ is the strict transform of $W$ on $Y$. 
In this case, we can easily check that $\mathbf B_{\tilde {Z}}=
(f_T)_*\Delta_T$ holds. We consider the $\mathbb R$-line bundle 
$\mathcal L$ on $X$ associated to $K_X+\Delta$. 
We fix an $\mathbb R$-Cartier $\mathbb R$-divisor 
$D_{\tilde {Z}}$ on $\tilde Z$ whose associated 
$\mathbb R$-line bundle is $\rho^*\nu^*(\mathcal L|_W)$. 
Then the trace $\mathbf M_{\tilde{Z}}$ of 
$\mathbf M$ on $\tilde{Z}$ is defined by 
$$\mathbf M_{\tilde{Z}}=
D_{\tilde Z}-K_{\tilde{Z}}-\mathbf B_{\tilde{Z}}.
$$
We simply write 
$$
\rho^*\nu^*(K_X+\Delta)=K_{\tilde Z}+\mathbf B_{\tilde Z} 
+\mathbf M_{\tilde Z}
$$ 
if there is no danger of confusion (see also Remark \ref{x-rem6.1}). 
\end{defn}

As we saw in Definition \ref{x-def1.3}, 
the $\mathbb{R}$-b-divisor $\mathbf B$ on $Z$ 
depends only on the singularities 
of $(X,\Delta)$ near $W$. 
Conversely, Theorem \ref{x-thm1.2} (ii)--(v) implies 
that $\mathbf B$ remembers properties of the singularities 
of $(X,\Delta)$ near $W$. 
If we put $B_Z=\mathbf B_Z$ and $M_Z=\mathbf M_Z$, 
then 
Theorem \ref{x-thm1.1} directly follows from Theorem \ref{x-thm1.2}. 
Our new formulation of adjunction and inversion of adjunction 
includes some classical results as special cases. 
The following corollary is the 
case of ${\rm dim}W={\rm dim}X-1$ which 
recovers the classical adjunction and inversion of adjunction. 

\begin{cor}[Classical Adjunction and Inversion of Adjunction]
\label{x-cor1.4} 
In Theorem \ref{x-thm1.1}, 
we further assume that $\dim W=\dim X-1$, that is, 
$W$ is a prime divisor on $X$. 
Then $M_Z$ and $B_Z$ become zero and 
Shokurov's different, respectively.  
Then (A) recovers Kawakita's inversion of adjunction 
on log canonicity. By (B), we have that 
$(X, \Delta)$ is purely log terminal in a neighborhood of 
$W$ if and only if $(Z, B_Z)$ is kawamata log terminal. 
\end{cor}

We know that we have already had many related results. 
We only make some remarks on \cite{hacon} and 
\cite{filipazzi}. 

\begin{rem}[Hacon's inversion of adjunction]\label{x-rem1.5}
In \cite[Theorem 1]{hacon}, Hacon treated inversion of adjunction 
on log canonicity for log canonical centers of arbitrary codimension 
under the extra assumption 
that $\Delta$ is a boundary $\mathbb Q$-divisor. 
We note that the b-divisor 
$\mathbf B(V; X, \Delta)$ in \cite{hacon} 
coincides with $\mathbf B$ in Theorem \ref{x-thm1.2} 
by \cite[Theorem 1.2]{fujino-hashizume-proc}. 
In \cite[Theorem 5.4]{fujino-hashizume}, we proved a generalization of 
\cite[Theorem 1]{hacon}. We note that 
$\mathbf B$ in \cite[Theorem 5.4]{fujino-hashizume} 
coincides with $\mathbf B$ in Theorem \ref{x-thm1.2}. Hence 
Theorem \ref{x-thm1.2} can be seen as a complete generalization of 
\cite[Theorem 5.4]{fujino-hashizume} and \cite[Theorem 1]{hacon}. 
\end{rem}

\begin{rem}[Generalized adjunction and 
inversion of adjunction by Filipazzi]\label{x-rem1.6} 
In \cite{filipazzi}, Filipazzi established some 
related results for 
generalized pairs (see, for example, \cite[Theorem 1.6]
{filipazzi}). 
Although they are more general than Theorems 
\ref{x-thm1.1} and \ref{x-thm1.2} in some sense, 
they do not include Theorem \ref{x-thm1.1}. 
\end{rem}

The main ingredients of Theorem \ref{x-thm1.2} are 
the existence theorem of log canonical modifications 
established in \cite{fujino-hashizume} and 
the theory of basic slc-trivial fibrations in \cite{fujino-slc-trivial} 
and \cite{fujino}. Hence this paper can 
be seen as a continuation of \cite{fujino} and 
\cite{fujino-hashizume}. 
Moreover, the theory of partial resolutions of singularities 
of pairs in \cite{bvp} is indispensable. 
We do not use Kawakita's inversion of adjunction (see \cite[Theorem]{kawakita}) 
nor 
the Kawamata--Viehweg vanishing theorem. 
If $K_X+\Delta$ is $\mathbb Q$-Cartier, then Theorem \ref{x-thm1.2} 
easily follows from \cite{fujino-slc-trivial}, \cite{fujino}, 
and \cite{fujino-hashizume}. 
Unfortunately, however, the framework of basic slc-trivial fibrations 
discussed in \cite{fujino-slc-trivial} is not sufficient 
for our purposes in this paper. 
Hence we establish the following result. 

\begin{thm}[Corollary \ref{x-cor5.2}]\label{x-thm1.7}
Let $f\colon (X, B)\to Y$ be a basic 
$\mathbb R$-slc-trivial fibration and 
let $\mathbf B$ and $\mathbf M$ be the 
discriminant and moduli $\mathbb R$-b-divisors associated 
to $f\colon (X, B)\to Y$, respectively. 
Then we have the following properties: 
\begin{itemize}
\item[(i)] $\mathbf K+\mathbf B$ is $\mathbb R$-b-Cartier, 
where $\mathbf K$ is the canonical 
b-divisor of $Y$, and 
\item[(ii)] $\mathbf M$ is b-potentially nef, that is,  
there exists a proper birational morphism $\sigma\colon Y'\to Y$ 
from a normal variety $Y'$ such that 
$\mathbf M_{Y'}$ is a potentially nef $\mathbb R$-divisor on $Y'$ and 
that $\mathbf M=\overline{\mathbf M_{Y'}}$ holds. 
\end{itemize}
\end{thm}

If $f\colon (X, B)\to Y$ is a basic $\mathbb Q$-slc-trivial 
fibration, then Theorem \ref{x-thm1.7} is nothing but 
\cite[Theorem 1.2]{fujino-slc-trivial}, which is the main result of 
\cite{fujino-slc-trivial}. 
More precisely, we establish: 
 
\begin{thm}[see Theorem \ref{x-thm5.1}]\label{x-thm1.8}
Let $f\colon (X, B)\to Y$ be a projective surjective morphism 
from a simple normal crossing pair $(X, B)$ to a 
smooth 
quasi-projective variety $Y$ such that 
every stratum of $X$ is dominant onto $Y$ and $f_*\mathcal O_X\simeq 
\mathcal O_Y$ with 
\begin{itemize}
\item $B=B^{\leq 1}$ holds over the generic point of $Y$, 
\item there exists an $\mathbb R$-Cartier $\mathbb R$-divisor 
$D$ on $Y$ such that $K_Y+B\sim _{\mathbb R}f^*D$ holds, and 
\item $\rank f_*\mathcal O_X(\lceil -(B^{<1})\rceil)=1$. 
\end{itemize} 
We assume that there exists a simple normal crossing 
divisor $\Sigma$ on $Y$ such that $\Supp D\subset \Sigma$ and 
that every stratum of $(X, \Supp B)$ is smooth over $Y\setminus \Sigma$. 
Let $\mathbf B$ and $\mathbf M$ be the discriminant and 
moduli $\mathbb R$-b-divisors associated to 
$f\colon (X, B)\to Y$, respectively. 
Then 
\begin{itemize}
\item[(i)] $\mathbf K+\mathbf B=
\overline {\mathbf K_Y+\mathbf B_Y}$ holds, 
where 
$\mathbf K$ is the canonical b-divisor of $Y$, and 
\item[(ii)] $\mathbf M_Y$ is a potentially 
nef $\mathbb R$-divisor on $Y$ with 
$\mathbf M=\overline {\mathbf M_Y}$. 
\end{itemize}
\end{thm}

Note that Theorem \ref{x-thm1.8} completely 
generalizes \cite[Lemma 2.8]{hu}. 
By Theorem \ref{x-thm1.8}, we can use the 
framework of basic slc-trivial fibrations 
in \cite{fujino-slc-trivial} 
for $\mathbb R$-divisors.  
We also note that the main part of this paper is devoted to 
the proof of Theorem \ref{x-thm1.8}. 
In the proof of Theorem \ref{x-thm1.2}, 
we naturally construct a basic $\mathbb R$-slc-trivial 
fibration $f\colon (V, \Delta_V)\to Z$ by taking a 
suitable resolution of singularities 
of the pair $(X, \Delta)$. 
The $\mathbb R$-b-divisors $\mathbf B$ and $\mathbf M$ on $Z$ 
in Theorem \ref{x-thm1.2} are the discriminant and 
moduli $\mathbb R$-b-divisors associated to $f\colon (V, \Delta_V)\to Z$, 
respectively.

\begin{conj}\label{x-conj1.9}
In Theorem \ref{x-thm1.8}, $\mathbf M_Y$ is semi-ample. 
\end{conj}

If Conjecture \ref{x-conj1.9} holds true, then $\mathbf M$ in 
Theorem \ref{x-thm1.2} is b-semi-ample, that is, 
$\mathbf M_{Z'}$ is semi-ample. 
Note that Conjecture \ref{x-conj1.9} follows from \cite[Conjecture 1.4]
{fujino-slc-trivial}. 
When $\dim Y=1$, we can easily check that 
$\mathbf M_Y$ is semi-ample by 
\cite[Corollary 1.4]{fujino-fujisawa-liu}. 
Unfortunately, however, it is still widely open. 
In this paper, we 
prove Conjecture \ref{x-conj1.9} for basic slc-trivial fibrations 
of relative dimension one under some extra assumption 
(see Theorem \ref{x-thm7.2}). 
Then we establish: 

\begin{thm}[see Corollary \ref{x-cor7.3}]\label{x-thm1.10} 
If $W$ is a codimension 
two log canonical center of $(X, \Delta)$ in Theorem \ref{x-thm1.2}, 
then $\mathbf M$ is b-semi-ample. 
\end{thm}

Theorem \ref{x-thm1.10} generalizes Kawamata's 
result (see \cite[Theorem 1]{kawamata}). 
For the details, see Corollary \ref{x-cor7.3}. 

\medskip

We briefly look at the organization of this paper. 
In Section \ref{x-sec2}, 
we recall some basic definitions and results. 
In Section \ref{x-sec3}, we introduce the notion of 
basic $\mathbb R$-slc-trivial fibrations and recall the main result of 
\cite{fujino-slc-trivial}. 
In Section \ref{x-sec4}, we slightly generalize 
the main result of \cite{fujino-slc-trivial}. 
This generalization (see Theorem \ref{x-thm4.1}) 
seems to be 
indispensable in order to treat basic $\mathbb R$-slc-trivial 
fibrations. 
In Section \ref{x-sec5}, we establish a fundamental theorem 
for basic $\mathbb R$-slc-trivial fibrations (see Theorems 
\ref{x-thm1.8} and \ref{x-thm5.1}). 
In Section \ref{x-sec6}, we prove the main result, 
that is, adjunction and inversion of adjunction for 
log canonical centers of arbitrary codimension, in full generality. 
More precisely, we first establish Theorem \ref{x-thm1.2}. 
Then we see that Theorem \ref{x-thm1.1} and Corollary \ref{x-cor1.4} 
easily follow from Theorem \ref{x-thm1.2}. 
In Section \ref{x-sec7}, we treat adjunction and inversion of adjunction 
for log canonical centers of codimension two. 

\begin{ack}
The first author was partially 
supported by JSPS KAKENHI Grant Numbers 
JP16H03925, JP16H06337, JP19H01787, JP20H00111, JP21H00974. 
The second author was partially supported by JSPS KAKENHI Grant
Numbers JP16J05875, JP19J00046. The authors thank 
Christopher Hacon very much for answering their question. 
\end{ack}

\section{Preliminaries}\label{x-sec2}

In this paper, we will freely use the standard notation as in 
\cite{fujino-fundamental}, \cite{fujino-foundations}, 
\cite{fujino-slc-trivial}, and 
\cite{fujino}. A {\em{scheme}} means 
a separated scheme of finite type over $\mathbb C$.  
A {\em{variety}} means an integral scheme, that is, 
an irreducible and reduced separated scheme of 
finite type over $\mathbb C$. 
We note that $\mathbb Q$ and 
$\mathbb R$ denote the sets of 
{\em{rational numbers}} and {\em{real numbers}}, respectively. 
We also note that $\mathbb Q_{>0}$ and $\mathbb R_{>0}$ 
are the sets of {\em{positive rational numbers}} and 
{\em{positive real numbers}}, respectively. 
Similarly, $\mathbb Q_{\geq 0}$ denotes the set of 
nonnegative rational numbers. 

Here, we collect some basic definitions for the reader's convenience. 
Let us start with the definition of {\em{potentially nef divisors}}. 

\begin{defn}[{Potentially nef divisors, see 
\cite[Definition 2.5]{fujino-slc-trivial}}]\label{x-def2.1} 
Let $X$ be a normal 
variety and let $D$ be a divisor on $X$. 
If there exist a completion $X^\dag$ of $X$, 
that is, $X^\dag$ is a complete normal 
variety and contains $X$ as a dense Zariski open subset, and 
a nef divisor $D^\dag$ on $X^\dag$ such that 
$D=D^\dag|_X$, then $D$ is called 
a {\em{potentially nef}} divisor on $X$. 
A finite $\mathbb Q_{>0}$-linear (resp.~$\mathbb R_{>0}$-linear) 
combination of potentially nef divisors is called 
a {\em{potentially nef}} $\mathbb Q$-divisor (resp.~$\mathbb R$-divisor). 
\end{defn}

We give two important remarks on potentially nef $\mathbb R$-divisors. 

\begin{rem}\label{x-rem2.2}
Let $D$ be a nef $\mathbb R$-divisor on a smooth 
projective variety $X$. 
Then $D$ is not necessarily a potentially nef $\mathbb R$-divisor. 
This means that $D$ is not always a finite $\mathbb R_{>0}$-linear 
combination of nef Cartier divisors on $X$. 
\end{rem}

\begin{rem}\label{x-rem2.3}
Let $X$ be a normal variety and let $D$ be a potentially 
nef $\mathbb R$-divisor on $X$. Then $D\cdot C\geq 0$ for every 
projective curve $C$ on $X$. 
In particular, $D$ is $\pi$-nef for every proper 
morphism $\pi\colon X\to S$ to a scheme $S$. 
\end{rem}

It is convenient to use {\em{b-divisors}} to explain 
several results. 
Here we do not repeat the definition of b-divisors. 
For the details, see \cite[Section 2]{fujino-slc-trivial}. 

\begin{defn}[Canonical b-divisors]\label{x-def2.4}
Let $X$ be a normal variety and let 
$\omega$ be a top rational differential 
form of $X$. 
Then $(\omega)$ defines a 
b-divisor $\mathbf K$. We call $\mathbf K$ 
the {\em{canonical b-divisor}} of $X$. 
\end{defn}

\begin{defn}[$\mathbb R$-Cartier closures]\label{x-def2.5}
The {\em{$\mathbb R$-Cartier 
closure}} of an $\mathbb R$-Cartier $\mathbb R$-divisor 
$D$ on a normal variety $X$ is the $\mathbb R$-b-divisor 
$\overline D$ with trace 
$$
\overline D _Y=f^*D, 
$$
where $f \colon Y\to X$ is a proper birational morphism 
from a normal variety $Y$. 
\end{defn}

We use the following definition in order to state  
our results (see Theorem \ref{x-thm1.2}). 

\begin{defn}[{\cite[Definition 2.12]{fujino-slc-trivial}}]\label{x-def2.6}
Let $X$ be a normal variety. 
An $\mathbb R$-b-divisor $\mathbf D$ of $X$ 
is {\em{b-potentially nef}} 
(resp.~{\em{b-semi-ample}}) if there 
exists a proper birational morphism $X'\to X$ from a normal 
variety $X'$ such that $\mathbf D=\overline {\mathbf D_{X'}}$, that 
is, $\mathbf D$ is the $\mathbb R$-Cartier 
closure of $\mathbf D_{X'}$, and that 
$\mathbf D_{X'}$ is potentially nef 
(resp.~semi-ample). 
An $\mathbb R$-b-divisor $\mathbf D$ 
of $X$ is {\em{$\mathbb R$-b-Cartier}} 
if there is a proper birational morphism $X'\to X$ from a normal 
variety $X'$ such that $\mathbf D=\overline{\mathbf D_{X'}}$. 
Obviously, $\mathbf D$ 
is said to be {\em{$\mathbb Q$-b-Cartier}} 
when $\mathbf D_{X'}$ is $\mathbb Q$-Cartier 
and 
$\mathbf D=\overline {\mathbf D_{X'}}$. 
\end{defn}

For the reader's convenience, let us recall the 
definition of singularities of pairs. The following definition is 
standard and is well known. 

\begin{defn}[Singularities of pairs]\label{x-def2.7}
Let $X$ be a variety and let $E$ be a prime divisor on $Y$ 
for some birational
morphism $f\colon Y\to X$ from a normal variety $Y$. 
Then $E$ is called a divisor {\em{over}} $X$. 
A {\em{normal pair}} $(X, \Delta)$ consists of 
a normal variety $X$ and an $\mathbb R$-divisor $\Delta$ on $X$ 
such that $K_X+\Delta$ is $\mathbb R$-Cartier. 
Let $(X, \Delta)$ be a normal pair and let 
$f\colon Y\to X$ be a projective 
birational morphism from a normal variety $Y$. 
Then we can write 
$$
K_Y=f^*(K_X+\Delta)+\sum _E a(E, X, \Delta)E
$$ 
with 
$$f_*\left(\underset{E}\sum a(E, X, \Delta)E\right)=-\Delta, 
$$ 
where $E$ runs over prime divisors on $Y$. 
We call $a(E, X, \Delta)$ the {\em{discrepancy}} of $E$ with 
respect to $(X, \Delta)$. 
Note that we can define the discrepancy $a(E, X, \Delta)$ for 
any prime divisor $E$ over $X$ by taking a suitable 
resolution of singularities of $X$. 
If $a(E, X, \Delta)\geq -1$ (resp.~$>-1$) for 
every prime divisor $E$ over $X$, 
then $(X, \Delta)$ is called {\em{sub log canonical}} (resp.~{\em{sub 
kawamata log terminal}}). 
We further assume that $\Delta$ is effective. 
Then $(X, \Delta)$ is 
called {\em{log canonical}} 
and {\em{kawamata log terminal}} 
if it is sub log canonical and sub kawamata log terminal, respectively. 
When $\Delta$ is effective and $a(E, X, \Delta)>-1$ holds 
for every exceptional divisor $E$ over $X$, we 
say that $(X, \Delta)$ is {\em{purely log terminal}}. 

Let $(X, \Delta)$ be a log canonical pair. If there 
exists a projective birational morphism 
$f\colon Y\to X$ from a smooth variety $Y$ such that 
both $\Exc(f)$, 
the exceptional locus of $f$,  
and  $\Exc(f)\cup \Supp f^{-1}_*\Delta$ are simple 
normal crossing divisors on $Y$ and that 
$a(E, X, \Delta)>-1$ holds for every 
$f$-exceptional divisor $E$ on $Y$, 
then $(X, \Delta)$ is 
called {\em{divisorial log terminal}} ({\em{dlt}}, for short). 
It is well known that if $(X, \Delta)$ 
is purely log terminal then 
it is divisorial log terminal. 
\end{defn}

In this paper, the notion of non-lc loci and log canonical 
centers is indispensable. 

\begin{defn}[Non-lc loci and log canonical centers]\label{x-def2.8}
Let $(X,\Delta)$ be a normal pair. 
If there exist a projective birational morphism 
$f\colon Y\to X$ from a normal variety $Y$ and a prime divisor $E$ on $Y$ 
such that $(X, \Delta)$ is 
sub log canonical in a neighborhood of the 
generic point of $f(E)$ and that 
$a(E, X, \Delta)=-1$, then $f(E)$ is called a 
{\em{log canonical center}} of 
$(X, \Delta)$. 

From now on, we further assume that 
$\Delta$ is effective. 
The {\em{non-lc locus}} of $(X,\Delta)$, 
denoted by $\Nlc(X, \Delta)$, is the smallest 
closed subset $Z$ of $X$ such that the 
complement $(X\setminus Z, \Delta|_{X\setminus Z})$ is log canonical. 
We can define a natural scheme structure on 
$\Nlc(X, \Delta)$ by the non-lc ideal sheaf 
$\mathcal J_{\NLC}(X, \Delta)$ of $(X, \Delta)$. 
For the definition of $\mathcal J_{\NLC}(X, \Delta)$, see 
\cite[Section 7]{fujino-fundamental}. 
\end{defn}

We omit the precise definition of {\em{NQC generalized 
log canonical pairs}} and {\em{NQC generalized kawamata 
log terminal pairs}} here since we need it only in 
Theorem \ref{x-thm1.1} and the statement of 
Theorem \ref{x-thm1.2} is sharper than that of 
Theorem \ref{x-thm1.1}. 
For the basic definitions and properties of {\em{generalized 
polarized pairs}}, we recommend the reader to see 
\cite[Section 2]{han-li}. Note that 
the notion of generalized pairs plays a crucial role 
in the recent study of higher-dimensional algebraic 
varieties. 

\begin{defn}\label{x-def2.9}
Let $X$ be an equidimensional reduced scheme. 
Note that $X$ is not necessarily regular in codimension one. 
Let $D$ be an $\mathbb R$-divisor (resp.~a $\mathbb Q$-divisor), 
that is, 
$D$ is a finite formal sum $\sum _i d_iD_i$, where 
$D_i$ is an irreducible reduced closed subscheme of $X$ of 
pure codimension one and $d_i\in \mathbb R$ 
(resp.~$d_i\in \mathbb Q$) for every $i$ 
such that $D_i\ne D_j$ for $i\ne j$. 
We put 
\begin{equation*}
D^{<1} =\sum _{d_i<1}d_iD_i, \quad 
D^{= 1}=\sum _{d_i= 1} D_i, \quad 
D^{>1} =\sum _{d_i>1}d_iD_i, \quad \text{and} \quad 
\lceil D\rceil =\sum _i \lceil d_i \rceil D_i, 
\end{equation*}
where 
$\lceil d_i\rceil$ is the integer defined by $d_i\leq 
\lceil d_i\rceil <d_i+1$. 
We note that $\lfloor D\rfloor =-\lceil -D\rceil$ and 
$\{D\}=D-\lfloor D\rfloor$. Similarly, we put 
$$
D^{\geq 1}=\sum _{d_i\geq 1}d_i D_i. 
$$

Let $D$ be an $\mathbb R$-divisor (resp.~a $\mathbb Q$-divisor) 
as above. 
We call $D$ a {\em{subboundary}} $\mathbb R$-divisor 
(resp.~$\mathbb Q$-divisor) 
if $D=D^{\leq 1}$ holds. 
When $D$ is effective and $D=D^{\leq 1}$ holds, 
we call $D$ a {\em{boundary}} 
$\mathbb R$-divisor (resp.~$\mathbb Q$-divisor). 

We further assume that 
$f\colon X\to Y$ is a surjective morphism onto a variety $Y$ 
such that every irreducible component of $X$ is dominant onto $Y$. 
Then we put 
$$
D^v=\sum _{f(D_i)\subsetneq Y}d_i D_i \quad \text{and}\quad 
D^h=\sum _{f(D_i)=Y}d_i D_i. 
$$
We call $D^v$ (resp.~$D^h$) the {\em{vertical part}} 
(resp.~{\em{horizontal part}})  
of $D$ 
with respect to $f\colon X\to Y$. 
\end{defn}

\section{On basic slc-trivial fibrations}\label{x-sec3}

Roughly speaking, a basic slc-trivial fibration is a canonical bundle 
formula for simple normal crossing pairs. 
It was first introduced in \cite{fujino-slc-trivial} based 
on \cite{fujino-fujisawa}. 
Let us start with the definition of simple normal crossing pairs. 

\begin{defn}[Simple normal crossing pairs]\label{x-def3.1}
A pair $(X, B)$ consists of 
an equidimensional reduced scheme $X$ and 
an $\mathbb R$-divisor $B$ on $X$. 
We say that the pair $(X, B)$ is {\em{simple 
normal crossing}} at a point $x\in X$ if 
$X$ has a Zariski open neighborhood 
$U$ of $x$ that can be embedded in a 
smooth variety $M$, 
where $M$ has a regular system of parameters 
$(x_1, \ldots, x_p, y_1, \ldots, y_r)$ at $x=0$ in which 
$U$ is defined by a monomial equation 
$$
x_1\cdots x_p=0
$$ 
and 
$$
B|_U=\sum _{i=1}^r \alpha_i (y_i=0)|_U, \quad \alpha_i\in \mathbb R. 
$$ 
We say that $(X, B)$ is a 
{\em{simple normal crossing pair}} if it is simple 
normal crossing at every point of $X$. 

Let $(X, B)$ be a simple normal crossing pair and let 
$\nu\colon X^\nu \to X$ be the normalization. 
We define $B^\nu$ by $K_{X^\nu}+B^\nu=\nu^*(K_X+B)$, that is, 
$B^\nu$ is the sum of the inverse images of $B$ and the singular locus of $X$. 
Then a {\em{stratum}} of $(X, B)$ is an irreducible component 
of $X$ or the $\nu$-image of some log canonical 
center of $(X^\nu, B^\nu)$. 
 
Let $(X, B)$ be a simple normal crossing pair and let 
$X=\bigcup_{i\in I} X_i$ be the irreducible decomposition of $X$. 
Then a {\em{stratum}} of $X$ means an irreducible 
component of $X_{i_1}\cap \cdots \cap X_{i_k}$ for 
some $\{i_1, \ldots, i_k\}\subset I$. 
It is easy to see that $W$ is a stratum of $X$ if and 
only if $W$ is a stratum of $(X, 0)$. 
\end{defn}

We introduce the notion of basic slc-trivial fibrations. 
In \cite{fujino-slc-trivial}, we only 
treat basic $\mathbb Q$-slc-trivial 
fibrations.  

\begin{defn}[{Basic slc-trivial fibrations, 
see \cite[Definition 4.1]{fujino-slc-trivial}}]\label{x-def3.2}
A {\em{pre-basic $\mathbb Q$-slc-trivial 
{\em{(}}resp.~$\mathbb R$-slc-trivial{\em{)}} fibration}} 
$f \colon (X, B)\to Y$ consists of 
a projective surjective morphism 
$f \colon X\to Y$ and a simple normal crossing pair $(X, B)$ satisfying 
the following properties: 
\begin{itemize}
\item[(1)] $Y$ is a normal variety,   
\item[(2)] every stratum of $X$ is dominant onto $Y$ and 
$f_*\mathcal O_X\simeq \mathcal O_Y$, 
\item[(3)] $B$ is a $\mathbb Q$-divisor 
(resp.~an $\mathbb R$-divisor) 
such that $B=B^{\leq 1}$ holds 
over 
the generic point of $Y$, and 
\item[(4)] there exists 
a $\mathbb Q$-Cartier $\mathbb Q$-divisor 
(resp.~an $\mathbb R$-Cartier $\mathbb R$-divisor) 
$D$ on $Y$ such that 
$
K_X+B\sim _{\mathbb Q}f^*D 
$ (resp.~$K_X+B\sim _{\mathbb R} f^*D$), 
that is, $K_X+B$ is $\mathbb Q$-linearly 
(resp.~$\mathbb R$-linearly) equivalent to $f^*D$.   
\end{itemize}
If a pre-basic $\mathbb Q$-slc-trivial 
(resp.~$\mathbb R$-slc-trivial) 
fibration $f \colon (X, B)\to Y$ also satisfies 
\begin{itemize}
\item[(5)] $\rank f_*\mathcal O_X(\lceil -(B^{<1})\rceil)=1$, 
\end{itemize}
then it is called a {\em{basic $\mathbb Q$-slc-trivial}} 
(resp.~{\em{$\mathbb R$-slc-trivial}}) {\em{fibration}}. 
\end{defn}

If there is no danger of confusion, 
we sometimes use {\em{{\em{(}}pre-{\em{)}}basic 
slc-trivial fibrations}} to denote 
{\em{{\em{(}}pre-{\em{)}}basic $\mathbb Q$-slc-trivial fibrations}} 
or {\em{{\em{(}}pre-{\em{)}}basic $\mathbb R$-slc-trivial fibrations}}. 

\begin{rem}[see Remark \ref{x-rem4.5}]\label{x-rem3.3}
The condition $f_*\mathcal O_X\simeq \mathcal O_Y$ in 
(2) in Definition \ref{x-def3.2} does not play an important role. 
Moreover, we have to treat the case where $\mathcal O_Y\subsetneq  
f_*\mathcal O_X$ in this paper. 
The reader can find that we do not need the condition 
$f_*\mathcal O_X\simeq \mathcal O_Y$ in many places 
in \cite{fujino-slc-trivial}. Hence it may be better to remove 
the condition $f_*\mathcal O_X\simeq \mathcal O_Y$ from the 
definition of pre-basic slc-trivial fibrations 
(see \cite[Definition 4.1]{fujino-slc-trivial} and Definition \ref{x-def3.2}). 
However, we keep it here not to cause unnecessary confusion. 

Note that the condition $f_*\mathcal O_X\simeq \mathcal O_Y$ 
always holds for basic slc-trivial fibrations even when we remove 
it from the definition of 
pre-basic slc-trivial fibrations.
We will see it more precisely. 
It is sufficient to see 
that if every stratum of $X$ is dominant onto $Y$ with 
$\rank f_*\mathcal O_X(\lceil -(B^{<1})\rceil)=1$ then 
the natural map $\mathcal O_Y\to f_*\mathcal O_X$ must be 
an isomorphism. 
We note that there are natural inclusions 
$$
\mathcal O_Y\hookrightarrow f_*\mathcal O_X\hookrightarrow 
f_*\mathcal O_X(\lceil -(B^{<1})\rceil)
$$ 
since $\lceil -(B^{<1})\rceil$ is effective. 
Hence $\mathcal O_Y\hookrightarrow f_*\mathcal O_X$ is an 
isomorphism over some nonempty Zariski open subset of $Y$ 
and $\rank f_*\mathcal O_X=1$ holds. 
We consider the Stein factorization 
$$
f\colon X\longrightarrow 
Z:=\Spec _Y f_*\mathcal O_X\overset {\alpha}{\longrightarrow} Y
$$
of $f\colon X\to Y$. Since every irreducible component of 
$X$ is dominant onto $Y$, $Z$ is a variety. 
Moreover, $\alpha\colon Z\to Y$ is birational since $\rank f_*\mathcal O_X=1$. 
By Zariski's main theorem, $\alpha\colon Z\to Y$ is an isomorphism. 
Hence the natural map $\mathcal O_Y\to f_*\mathcal O_X$ is 
an isomorphism. 
\end{rem}

In order to define discriminant $\mathbb R$-b-divisors and 
moduli $\mathbb R$-b-divisors for basic slc-trivial fibrations, 
we need the notion 
of {\em{induced $($pre-$)$basic slc-trivial fibrations}}. 

\begin{defn}[{Induced (pre-)basic slc-trivial 
fibrations, \cite[4.3]{fujino-slc-trivial}}]\label{x-def3.4}
Let $f \colon (X, B)\to Y$ be a 
(pre-)basic slc-trivial fibration 
and let $\sigma \colon Y'\to Y$ be a generically finite 
surjective morphism from a normal variety $Y'$. 
Then we have 
an {\em{induced {\em{(}}pre-{\em{)}}basic slc-trivial fibration}} 
$f' \colon (X', B_{X'})\to Y'$, where 
$B_{X'}$ is defined by $\mu^*(K_X+B)=K_{X'}+B_{X'}$, with 
the following commutative diagram: 
$$
\xymatrix{
   (X', B_{X'}) \ar[r]^{\mu} \ar[d]_{f'} & (X, B)\ar[d]^{f} \\
   Y' \ar[r]_{\sigma} & Y, 
} 
$$
where $X'$ coincides with 
$X\times _{Y}Y'$ over a nonempty Zariski open subset of $Y'$. 
More precisely, $(X', B_{X'})$ is a simple 
normal crossing pair with a morphism 
$X'\to X\times _Y Y'$ that is an isomorphism over 
a nonempty Zariski open subset of $Y'$ such that 
$X'$ is projective over $Y'$ and that every 
stratum of $X'$ is dominant onto 
$Y'$. 
\end{defn}

Now we are ready to define {\em{discriminant 
$\mathbb R$-b-divisors}} and 
{\em{moduli $\mathbb R$-b-divisors}} for 
basic slc-trivial fibrations. 

\begin{defn}[{Discriminant and 
moduli $\mathbb R$-b-divisors, \cite[4.5]{fujino-slc-trivial}}]\label{x-def3.5} 
Let $f \colon (X, B)\to Y$ be a (pre-)basic 
slc-trivial fibration as in Definition \ref{x-def3.2}. 
Let $P$ be a prime divisor on $Y$. 
By shrinking $Y$ around the generic point of $P$, 
we assume that $P$ is Cartier. We set 
$$
b_P=\max \left\{t \in \mathbb R\, \left|\, 
\begin{array}{l}  {\text{$(X^\nu, B^\nu+t\nu^*f^*P)$ is sub log canonical}}\\
{\text{over the generic point of $P$}} 
\end{array}\right. \right\},  
$$ 
where $\nu \colon X^\nu\to X$ is the normalization and 
$K_{X^\nu}+B^\nu=\nu^*(K_X+B)$, that is, 
$B^\nu$ is the sum of the inverse images of $B$ and the singular 
locus of $X$, and 
set $$
B_Y=\sum _P (1-b_P)P, 
$$ 
where $P$ runs over prime divisors on $Y$. 
Then it is easy to  see that 
$B_Y$ is a well-defined $\mathbb R$-divisor on 
$Y$ and is called the {\em{discriminant 
$\mathbb R$-divisor}} of $f \colon (X, B)\to Y$. We set 
$$
M_Y=D-K_Y-B_Y
$$ 
and call $M_Y$ the {\em{moduli $\mathbb R$-divisor}} of $f \colon 
(X, B)\to Y$. 
By definition, we have 
$$
K_X+B\sim _{\mathbb R}f^*(K_Y+B_Y+M_Y). 
$$

Let $\sigma\colon Y'\to Y$ be a proper birational morphism 
from a normal variety $Y'$ and let $f' \colon (X', B_{X'})\to Y'$ be 
an induced (pre-)basic slc-trivial fibration 
by $\sigma \colon Y'\to Y$.  
We can define $B_{Y'}$, $K_{Y'}$ and $M_{Y'}$ such that 
$\sigma^*D=K_{Y'}+B_{Y'}+M_{Y'}$, 
$\sigma_*B_{Y'}=B_Y$, $\sigma _*K_{Y'}=K_Y$ 
and $\sigma_*M_{Y'}=M_Y$. We note that 
$B_{Y'}$ is independent of the choice of $(X', B_{X'})$, 
that is, $B_{Y'}$ is well defined. Hence 
there exist a unique $\mathbb R$-b-divisor $\mathbf B$ 
such that 
$\mathbf B_{Y'}=B_{Y'}$ for every 
$\sigma \colon Y'\to Y$ and a unique 
$\mathbb R$-b-divisor $\mathbf M$ 
such that $\mathbf M_{Y'}=M_{Y'}$ for 
every $\sigma \colon Y'\to Y$. 
Note that $\mathbf B$ is called the 
{\em{discriminant $\mathbb R$-b-divisor}} and 
that $\mathbf M$ is called 
the {\em{moduli $\mathbb R$-b-divisor}} associated 
to $f \colon (X, B)\to Y$. 
We sometimes simply say that $\mathbf M$ is 
the {\em{moduli part}} of $f \colon (X, B)\to Y$. 
\end{defn}

Let $g\colon V\to Y$ be a proper surjective morphism 
from an equidimensional normal 
scheme $V$ onto a normal variety $Y$ 
such that every irreducible component of $V$ is 
dominant onto $Y$. 
Let $G$ 
be an $\mathbb R$-divisor on $V$ such that 
$K_V+G$ is $\mathbb R$-Cartier. Assume that 
$(V, G)$ is sub log canonical over the generic point of $Y$. 
Let $\sigma\colon Y'\to Y$ be a generically 
finite surjective morphism 
from a normal variety $Y'$. 
Then we have the following commutative diagram: 
$$
\xymatrix{
   (V', G') \ar[r]^{\mu} \ar[d]_{g'} & (V, G)\ar[d]^{g} \\
   Y' \ar[r]_{\sigma} & Y, 
} 
$$ 
where $V'$ is the normalization of the main 
components of $V\times _Y Y'$ 
and $G'$ is defined by 
$K_{V'}+G'=\mu^*(K_V+G)$. 
Then we can define the discriminant $\mathbb R$-divisor 
$B_Y$ on $Y$ and the discriminant 
$\mathbb R$-b-divisor $\mathbf B$ as in Definition \ref{x-def3.5}. 
Let $f\colon (X, B)\to Y$ be a (pre-)basic slc-trivial fibration and let 
$\nu\colon X^\nu\to X$ be the normalization with 
$K_{X^\nu}+B^\nu=\nu^*(K_X+B)$. 
Then the discriminant $\mathbb R$-b-divisor $\mathbf B$ 
associated to $f\colon (X, B)\to Y$ defined in Definition \ref{x-def3.5} 
obviously coincides with 
that of $f\circ \nu\colon (X^\nu, B^\nu)\to Y$ by definition. 

\medskip

Let us see the main result of \cite{fujino-slc-trivial}. 

\begin{thm}[{\cite[Theorem 1.2]{fujino-slc-trivial}}]\label{x-thm3.6}
Let $f\colon (X, B)\to Y$ be a basic $\mathbb Q$-slc-trivial fibration and 
let $\mathbf B$ and $\mathbf M$ be the 
discriminant and moduli $\mathbb Q$-b-divisors 
associated to $f\colon (X, B)\to Y$, respectively. 
Then we have the following properties: 
\begin{itemize}
\item[(i)] $\mathbf K+\mathbf B$ is $\mathbb Q$-b-Cartier, 
where $\mathbf K$ is the canonical 
b-divisor of $Y$, and 
\item[(ii)] $\mathbf M$ is b-potentially nef, that is,  
there exists a proper birational morphism $\sigma\colon Y'\to Y$ 
from a normal variety $Y'$ such that 
$\mathbf M_{Y'}$ is a potentially nef $\mathbb Q$-divisor on $Y'$ and 
that $\mathbf M=\overline{\mathbf M_{Y'}}$. 
\end{itemize}
\end{thm}

The following result was established in \cite{fujino-fujisawa-liu}. 

\begin{thm}[{\cite[Corollary 1.4]{fujino-fujisawa-liu}}]\label{x-thm3.7} 
In Theorem \ref{x-thm3.6}, if $Y$ is a curve, 
then $\mathbf M_Y$ is semi-ample. 
\end{thm}

We close this section with important remarks 
on \cite{fujino-slc-trivial}. 

\begin{rem}\label{x-rem3.8}
In (d) in \cite[Section 6]{fujino-slc-trivial}, we 
assume that $\Supp M_Y \subset \Supp \Sigma_Y$. 
However, this conditions is unnecessary. This is 
because if $P$ is not an irreducible component 
of $\Supp \Sigma_Y$ then we can always take a prime 
divisor $Q$ on $V$ such that $\mult _Q(-B_V+h^*B_Y)=0$, 
$h(Q)=P$, and $\mult _Q h^*P=1$ 
(see \cite[Proposition 6.3 (iv)]{fujino-slc-trivial}). 
\end{rem}

\begin{rem}\label{x-rem3.9}
In \cite[6.1]{fujino-slc-trivial}, 
we assume that $\Supp \left( B-f^*(B_Y+M_Y)\right)$ 
is a simple normal crossing divisor on $X$. 
However, we do not need this assumption. 
All we need in \cite[6.1]{fujino-slc-trivial} is 
the fact that the support of $\{\Delta\}$ 
is a simple normal crossing divisor on $X$. 
We note that 
$$\Supp \{\Delta\}\subset \Supp \left( B-f^*(B_Y+M_Y)\right)$$ 
always holds since $\Delta =K_{X/Y}+B-f^*(B_Y+M_Y)$.  
\end{rem}

\section{Fundamental theorem for basic $\mathbb Q$-slc-trivial 
fibrations}\label{x-sec4}

In this section, we will slightly generalize 
the main theorem of \cite{fujino-slc-trivial} (see Theorem 
\ref{x-thm3.6}). The following theorem is the main result of this 
section. 

\begin{thm}[{see \cite[Theorem 1.2]{fujino-slc-trivial}}]\label{x-thm4.1}
Let $f\colon (X, B)\to Y$ be a basic 
$\mathbb Q$-slc-trivial fibration such 
that $Y$ is a smooth quasi-projective variety. 
We write 
$K_X+B\sim _{\mathbb Q} f^*D$. 
Assume that 
there exists a simple normal crossing divisor 
$\Sigma$ on $Y$ such that 
$\Supp D\subset \Sigma$ and 
that every stratum of $(X, \Supp B)$ is 
smooth over $Y\setminus \Sigma$. 
Then 
\begin{itemize}
\item[(i)] $\mathbf K +\mathbf B=
\overline {\mathbf K_Y+\mathbf B_Y}$ holds, and 
\item[(ii)] $\mathbf M_Y$ is a potentially nef $\mathbb Q$-divisor 
on $Y$ with $\mathbf M=\overline {\mathbf M_Y}$. 
\end{itemize}
\end{thm}

In Section \ref{x-sec5}, Theorem \ref{x-thm4.1} will be generalized 
for basic $\mathbb R$-slc-trivial fibrations (see Theorems 
\ref{x-thm1.8} and \ref{x-thm5.1}). 
We note that Theorem \ref{x-thm4.1} is indispensable for the 
proof of Theorem \ref{x-thm5.1} in Section \ref{x-sec5}. 
For the proof of 
Theorem \ref{x-thm4.1}, 
we prepare a lemma on simultaneous partial 
resolutions of singularities 
of pairs. 
Let us recall the main result of 
\cite{bvp}. 

\begin{thm}[{\cite[Theorem 1.4]{bvp}}]\label{x-thm4.2}
Let $X$ be a reduced scheme, and let $D$ be a 
$\mathbb{Q}$-divisor on $X$. 
Let $U$ be the largest open subset of $X$ such 
that $(U,D|_{U})$ is a simple normal crossing pair. 
Then there is a morphism 
$f\colon \tilde{X} \to X$, which is a 
composition of blow-ups, such that 
\begin{itemize}
\item
the exceptional locus $\Exc(f)$ is of 
pure codimension one, 
\item
putting $\tilde{D}=f^{-1}_{*}D+\Exc(f)$ 
then $(\tilde{X},\tilde{D})$ is a simple normal crossing pair, and 
\item
$f$ is an isomorphism over $U$.
\end{itemize}
\end{thm}

\begin{rem}[{Functoriality, see \cite[Remark 1.5 (3)]{bvp}}]\label{x-rem4.3}
By \cite[Remark 1.5 (3)]{bvp}, for every 
reduced scheme $X$ and a $\mathbb{Q}$-divisor $D$ on $X$ we may take 
$f_{X}\colon \tilde{X}\to X$ of Theorem \ref{x-thm4.2} 
satisfying the following funtoriality.  
Suppose that we are given an \'etale or 
a smooth morphism $\phi \colon X\to Y$ 
of reduced schemes and $\mathbb{Q}$-divisors 
$D_{X}$ and $D_{Y}$ on $X$ and $Y$ respectively such that 
\begin{itemize}
\item $\phi^*D_Y=D_X$, and 
\item
 the number of irreducible 
 components of $X$ (resp.~$\Supp D_{X}$) at 
 a point $x\in X$ coincides with that of $Y$ (resp.~$\Supp D_{Y}$) 
 at $\phi(x)\in Y$ for every $x\in X$.
\end{itemize}
Then, the morphisms $f_{X}\colon \tilde{X}\to X$ and 
$f_{Y}\colon \tilde{Y}\to Y$ as in 
Theorem \ref{x-thm4.2} form the diagram of the fiber product
$$
\xymatrix{
\tilde{X} \ar[d]_{f_{X}}\ar[r]^{\tilde{\phi}}& \tilde{Y} \ar[d]^{f_{Y}}\\ 
X \ar[r]_{\phi} &Y \ar@{}[lu]|{\square},
}
$$
that is, $\tilde{X}=X\times_{Y}\tilde{Y}$.
\end{rem}

The following lemma is a key lemma for the 
proof of Theorem \ref{x-thm4.1}. 

\begin{lem}\label{x-lem4.4}
Let $(X, B)$ be a simple normal crossing pair such that 
$B$ is a $\mathbb Q$-divisor. 
Let $f\colon X\to Y$ be a surjective morphism 
onto a smooth variety $Y$ such that 
every stratum of $(X, \Supp B)$ is smooth over $Y$. 
We put $\Delta=K_X+B$ and 
assume that $b\Delta\sim 0$ for some positive integer $b$.
We consider a $b$-fold cyclic cover 
$$
\pi\colon \widetilde X=\Spec _X 
\bigoplus _{i=0}^{b-1} \mathcal O_X(\lfloor i\Delta\rfloor) \longrightarrow X$$ 
associated to $b\Delta\sim 0$. 
We put $K_{\widetilde X}+B_{\widetilde X}=\pi^*(K_X+B)$. 
Let $\widetilde U$ be the largest Zariski open subset 
of $\widetilde X$ such that 
$(\widetilde U, B_{\widetilde X}|_{\widetilde U})$ is a simple 
normal crossing pair. 
Then there exists a morphism 
$d\colon V\to \widetilde X$ given by a composition of 
blow-ups such that 
\begin{itemize}
\item[(i)] $d$ is an isomorphism over $\widetilde U$, 
\item[(ii)] $(V, B_V)$ is a simple normal crossing pair, 
where $K_V+B_V=d^*(K_{\widetilde X}+B_{\widetilde X})$, 
and 
\item[(iii)] every stratum of $(V, \Supp B_V)$ is smooth 
over $Y$. 
\end{itemize}
\end{lem}
\begin{proof}
Let us quickly recall the $b$-fold 
cyclic cover $\pi\colon \widetilde X
\to X$. 
We fix a rational function $\phi$ 
on $X$ such that $b\Delta=
\ddiv(\phi).$
As usual, we can define an $\mathcal O_X$-algebra 
structure of 
$\bigoplus _{i=0}^{b-1}\mathcal O_X(\lfloor i\Delta\rfloor)$ 
by $b\Delta={\ddiv}(\phi)$. 
We note that 
$$
\mathcal O_X(\lfloor i\Delta\rfloor)\times 
\mathcal O_X(\lfloor j\Delta\rfloor) \to \mathcal O_X(\lfloor 
(i+j)\Delta\rfloor)
$$ 
is well defined for $0\leq i, j \leq b-1$ by 
$\lfloor i\Delta\rfloor +\lfloor j\Delta\rfloor \leq 
\lfloor (i+j)\Delta\rfloor$ and 
that 
$$
\mathcal O_X(\lfloor (i+j)\Delta\rfloor)\simeq 
\mathcal O_X(\lfloor (i+j-b)\Delta\rfloor)
$$ 
for $i+j\geq b$ defined by the multiplication with $\phi^{-1}$. 
We put 
$$
\pi\colon \widetilde X=\Spec _X 
\bigoplus _{i=0}^{b-1} \mathcal O_X(\lfloor i\Delta\rfloor)
$$ 
and call it a $b$-fold cyclic cover associated to $b\Delta\sim 0$. 
By construction, $\pi\colon\widetilde X\to X$ is \'etale 
outside $\Supp \{\Delta\}$. 
We note that $\widetilde X$ is normal 
over a neighborhood of the generic point of every irreducible 
component of $\Supp \{\Delta\}$. 
We also note that $(\widetilde X, B_{\widetilde X})$ is 
simple normal crossing in codimension one. 
Throughout this proof, we will freely use 
the following commutative diagram. 
\begin{equation*}
\xymatrix{
(X, B)\ar[d]_-f & (\widetilde X, B_{\widetilde X}) \ar[dl]_-{\widetilde f}
\ar[l]_-\pi& (V, B_V)\ar[dll]^-h\ar[l]_-d\\ 
Y & & 
}
\end{equation*}
\setcounter{step}{0}
\begin{step}\label{x-step4.4.1}
Let $U$ and $Z$ be affine open neighborhood 
of $x\in X$ and $y=f(x)\in Y$, respectively. 
Without loss of generality, we may assume that 
$U$ is a simple normal crossing divisor on a smooth 
affine variety $W$ since $(X, B)$ is a simple 
normal crossing pair. 
By shrinking $W$, $U$, and $Z$ suitably, 
we get the following commutative diagram 
$$
\xymatrix{
U \ar[dr]_-{f|_U}\ar@{^{(}->}[r]^-\iota& W \ar[d]^-p\\ 
& Z
}
$$
where $\iota$ is the natural closed embedding $U\hookrightarrow 
W$. 
From now on, we will repeatedly shrink $W$, 
$U$, and $Z$ suitably without mentioning it explicitly. 
Since every stratum of $X$ is smooth over $Y$, 
we may assume that 
$p$ is a smooth morphism between smooth affine varieties. 
\end{step}
\begin{step}\label{x-step4.4.2}
Since $p\colon W\to Z$ is a smooth morphism, 
there exists a commutative diagram 
$$
\xymatrix{
W \ar[rd]_-{p}\ar[r]^-g& Z\times \mathbb C^n\ar[d]^-{p_1}\\ 
& Z
}
$$ 
where $g$ is \'etale and $p_1$ is the first projection 
(see, for example, \cite[Chapter VII, Definition (1.1) and 
Theorem (1.8)]{altman-kleiman}). 
By choosing a coordinate system $(z_1, \ldots, z_n)$ 
of $\mathbb C^n$ suitably and shrinking 
$U$ and $W$ if necessary, 
we may further assume that 
$U$ is defined by a monomial 
$$
x_1\cdots x_p=0
$$ 
on $W$, where $x_i=g^*z_i$ for $1\leq i\leq p$, and 
$$
B|_{U}=\sum _{i=1}^r\alpha _i (y_i=0)|_U \quad \text{with} \quad 
\alpha _i \in \mathbb Q
$$ 
holds, where $y_i=g^*z_{p+i}$ for $1\leq i\leq r$. 
Here, we used the hypothesis 
that every stratum of $(X, \Supp B)$ is smooth over $Y$.
\end{step}
\begin{step}\label{x-step4.4.3}
We put 
$L=(z_1\cdots z_p=0)$ in $\mathbb C^n$. 
Then we have the following commutative diagram. 
$$
\xymatrix{
U \ar[dr]_-{f|_U}\ar[r]^-{g|_U}& Z\ar[d]^-{p_1}\times L \\ 
& Z
}
$$ 
Note that $g|_{U}$ is \'etale because it is 
the base change of $g$ by $L\hookrightarrow \mathbb C^n$. 
We put 
$$D=\sum _{i=1}^r \alpha_i (z_{p+i}=0)$$ on $\mathbb C^n$. 
Let $p_2\colon Z\times \mathbb C^n\to 
\mathbb C^n$ be the second projection. 
Then $B|_{U}=g^*p^*_2D|_U$ holds. 
\end{step}
\begin{step}\label{x-step4.4.4}
Without loss of generality, we may assume that 
$K_Z\sim 0$ by shrinking $Z$ suitably. 
Then $K_U\sim 0$ holds. Hence, by using 
the second projection $p_2\colon Z\times \mathbb C^n 
\to \mathbb C^n$, we have 
$$ 
0\sim 
b\Delta|_U
=b(K_U+B|_U)\sim 
bg|^*_U\bigl(p_2^*D|_{Z\times L}\bigr). 
$$ 
Since $g|_{U}$ is \'etale, we see that all 
the coefficients of $bp_{2}^{*}D|_{Z\times L}$ are integers. 
Since $p_{2}$ is the second projection 
and $D+L$ is a simple normal crossing 
divisor on $\mathbb{C}^{n}$, all the coefficients of $bD$ are integers. 
Therefore, we have $bD\sim0$. 
We fix a rational function $\sigma$ 
on $\mathbb{C}^{n}$ such that $bD=\ddiv(\sigma)$. 
We consider the $b$-fold cyclic cover 
$\alpha\colon M\to \mathbb C^n$ associated to 
$bD=\ddiv(\sigma)$. 
We put $N=\alpha^{-1}L$. 
We define $B_N$ by $K_N+B_N=(\alpha|_{N})^*(K_L+D|_L)$ 
and put $B_{Z\times N}=p^*_2B_N$, 
where $p_2\colon Z\times N\to N$ is the second projection. 
Then we get the following commutative diagram:  
$$
\xymatrix{
U'\ar[d]\ar[r]^-{g'} & Z\times N\ar[d]^-{{\rm id}_Z\times 
(\alpha|_N)}\ar[r]^-{p_2}&N \ar@{^{(}->}[r]\ar[d]^-{\alpha|_N}
&M\ar[d]^-{\alpha} \\ 
U\ar[dr]_-{f|_U} \ar[r]^-{g|_U} &
Z\times L \ar[d]^-{p_1}\ar[r]_-{p_2}&L \ar@{^{(}->}[r]&\mathbb C^n\\ 
& Z
}
$$
where $g'\colon U'\to Z\times N$ is the base change of $g|_U\colon 
U\to Z\times L$ by ${\rm id}_Z\times (\alpha|_N)$. 
We put $B_{U'}=g'^*B_{Z\times N}$. 
Then $K_{U'}+B_{U'}$ is equal to 
the pullback of $K_{U}+B|_{U}$ to $U'$. 
\end{step}
\begin{step}\label{x-step4.4.5}
Since $\alpha\colon M\to \mathbb C^n$ is the 
$b$-fold cyclic cover associated to $bD=\ddiv(\sigma)$, we see that
$$M=\Spec _{\mathbb C^n}
\bigoplus _{i=0}^{b-1} \mathcal O_{\mathbb C^n}(\lfloor iD\rfloor).$$
Since $p_{2}\circ g\colon W\to Z\times \mathbb C^n\to \mathbb C^{n}$ 
is the composition of an \'etale morphism and 
the second projection, we have $g^{*}p_{2}^{*}\lfloor iD\rfloor|_{U}
=\lfloor ig^{*}p_{2}^{*}D\rfloor|_{U}=\lfloor iB\rfloor|_{U}
=\lfloor iB|_{U}\rfloor$, where the last equality follows 
from that $(U,B|_{U})$ is a simple normal crossing pair. 
Let $\sigma_{U}$ be a rational function 
on $U$ which is the pullback of $\sigma$.  
Then $bB|_{U}=\ddiv(\sigma_{U})$ because we have $bD=\ddiv(\sigma)$.  
By the construction of $U'\to U$, we see that
$$U'=\Spec _{U}
\bigoplus _{i=0}^{b-1} \mathcal O_U(\lfloor iB|_{U}\rfloor)$$
and $U'\to U$ is the $b$-fold cyclic cover 
associated to $bB|_{U}=\ddiv(\sigma_{U})$. 

We recall that $\Delta=K_{X}+B$ 
and $\widetilde X\to X$ is the $b$-fold cyclic 
cover associated to $b\Delta=\ddiv(\phi)$. 
We put $\phi_{U}$ as the restriction of $\phi$ to $U$.
Then, the morphism $\pi^{-1}(U)\to U$ is 
the $b$-fold cyclic cover associated to $b\Delta|_{U}=\ddiv(\phi_{U})$. 
Now $\Delta|_{U}-B|_{U}$ is a Cartier 
divisor on $U$ and $b(\Delta|_{U}-B|_{U})=
\ddiv(\phi_{U}\cdot\sigma_{U}^{-1})$.  
With this relation, we construct a $b$-fold 
cyclic cover $\tau\colon \overline{U}\to U$. 
Then $\tau$ is \'etale, $\tau^{*}(\Delta|_{U}-B|_{U})$ 
is Cartier and $\tau^{*}(\Delta|_{U}-B|_{U})\sim 0$. 
So there exists a rational function $\xi$ 
on $\overline{U}$ such that $\xi^{b}
=\tau^*(\phi_U\cdot \sigma^{-1}_U)$, 
equivalently, $\tau^{*}(\Delta|_{U}-B|_{U})
=\ddiv(\xi)$. 
From this, the $b$-fold cyclic cover 
$U^\dag _1\to \overline U$ associated to $b\tau^{*}\Delta|_{U}=
\ddiv(\tau^{*}\phi_{U})$ is isomorphic to 
the $b$-fold cyclic cover $U^\dag_2\to \overline U$ 
associated to $b\tau^{*}B|_{U}
=\ddiv(\tau^{*}\sigma_{U})=
\ddiv(\tau^{*}\phi_{U}\cdot \xi^{-b})$. 
Since $\tau\colon \overline{U}\to U$ 
is \'etale, the construction of $U^\dag_2$ 
shows that $U_2^\dag \to \overline{U}$ is 
the base change of $U'\to U$ by $\overline{U}\to U$. 
Similarly, we see that $U^\dag_1 \to \overline{U}$ 
is the base change of $\pi^{-1}(U)\to U$ by $\overline{U}\to U$. 
$$
\xymatrix{
U^\dag_2 \ar[d]\ar[r]^-{a_2}& U' \ar[d]&&  U^\dag_1 
\ar[d]\ar[r]^-{a_1}& \pi^{-1}(U) \ar[d]\\ 
\overline{U} \ar[r]_-\tau 
&U \ar@{}[lu]|{\Box}&&  \overline{U} \ar[r]_-\tau &U \ar@{}[lu]|{\Box}
}
$$
We put $a_1\colon U^\dag_1\to \pi^{-1}(U)$ and $a_2\colon U^\dag_2\to U'$. 
By construction, $a_1$ and $a_2$ 
are \'etale. 
We see that 
the composition 
$U^\dag_1 \to \pi^{-1}(U)\to U$ 
is isomorphic to the composition 
$U^\dag_2 \to U'\to U$ by construction. 
By this isomorphism, we obtain that 
$a_1^*(B_{\widetilde X}|_{\pi^{-1}(U)})$ is 
isomorphic to $a_2^*B_{U'}$. 

In this way, there exist \'etale morphisms $a\colon U^\dag\to \pi^{-1}(U)$ 
and $a'\colon U^\dag\to U'$ over $Z$ 
such that $U^\dag_1\simeq 
U^\dag\simeq U^\dag_2$ 
with the following commutative diagram: 
$$
\xymatrix{
&U^\dag \ar[dl]_-a\ar[dr]^-{a'}& \\ 
\pi^{-1}(U)  \ar[d]_-{\pi} & & U' \ar[d] \\ 
U \ar@{=}[rr]&  &U
}
$$
such that $a^*(B_{\widetilde X}|_{\pi^{-1}(U)})=a'^*B_{U'}$. 
\end{step}
\begin{step}\label{x-step4.4.6}
We apply \cite[Theorem 1.4]{bvp} (see Theorem \ref{x-thm4.2}) 
to the pair 
$(\widetilde X, B_{\widetilde X})$. 
Then we obtain a morphism $d\colon V\to \widetilde X$ given 
by a composite of blow-ups satisfying (i) and (ii). 
Hence, all we have to do is to check that 
$d\colon V\to \widetilde X$ satisfies (iii). 
\end{step}
\begin{step}\label{x-step4.4.7}
Recall that $B_{Z\times N}=p^*_2B_N$, 
where $p_2\colon Z\times N \to N$, and $B_{U'}=g'^*B_{Z\times N}$. 
Recall also the relation $a^*(B_{\widetilde X}|_{\pi^{-1}(U)})=a'^*B_{U'}$. 
We apply \cite[Theorem 1.4]{bvp} (see Theorem \ref{x-thm4.2}) to 
$N$ and $B_N$, and we obtain a morphism 
$\beta\colon N'\to N$ given by a composition of 
blow-ups. 
We apply \cite[Theorem 1.4]{bvp} again to $Z\times N$ and $B_{Z \times N}$. 
Then we get a morphism 
${\rm id}_Z\times \beta\colon Z\times N'\to Z\times 
N$ by the functoriality of \cite[Theorem 1.4]{bvp} 
(see Remark \ref{x-rem4.3}). 
We put $\widehat{V}=d^{-1}(\pi^{-1}(U)) \subset V$, and 
we apply \cite[Theorem 1.4]{bvp} to the pair of $U^\dag$ 
and $a^*(B_{\widetilde X}|_{\pi^{-1}(U)})$, and the pair of $U'$ and $B_{U'}$. 
Then we obtain morphisms $V^\dag \to U^\dag$ and $V'\to U'$. 

We check that we may apply the 
functoriality (see Remark \ref{x-rem4.3}) to
 the morphisms 
 $$\text{$g'\colon U' \to Z\times N$, $a'\colon U^{\dag} \to U'$, 
 and $a\colon U^{\dag} \to \pi^{-1}(U)$}$$
  (see the diagram in the next paragraph) and divisors 
$$\text{$B_{Z\times N}$ on $Z\times N$, 
$B_{U'}$ on $U'$, and 
$B_{\widetilde X}|_{\pi^{-1}(U)}$ on $\pi^{-1}(U)$ and their pullbacks.}$$ 
We only check the second condition of 
Remark \ref{x-rem4.3} for schemes 
because the case of divisors can be proved by the same way.
By construction, $g'$ is the base change 
of $g|_{U}\colon U \to Z\times L$ by the 
morphism $Z\times N \to Z \times L$. 
Because $Z \times L$ is a simple normal 
crossing divisor on $Z\times \mathbb C^n$ 
and $g|_{U}$ is \'etale, 
by arguing locally, we see that $g|_{U}$ 
satisfies the second condition of Remark \ref{x-rem4.3}. 
Then so does $g'$ since $g'$ is constructed by the base change of $g|_{U}$. 
Similarly, $a'$ (resp.~$a$) is constructed 
with the base change of 
$\tau\colon \overline{U}\to U$ by 
$U' \to U$ (resp.~$\pi^{-1}(U)\to U$), 
and $U$ is a simple normal crossing divisor 
on $W$. 
Thus, the same argument as above 
implies that $a'$ and $a$ satisfy the 
second condition of Remark \ref{x-rem4.3}. 
Thus, we may apply the 
functoriality (see Remark \ref{x-rem4.3}) to the above morphisms and divisors. 

Applying the functoriality (see Remark \ref{x-rem4.3}), we have 
the following diagram: 
$$
\xymatrix{
\widehat{V} \ar[d]_-{d|_{\widehat{V}}}& 
V^\dag \ar[l]\ar[d]\ar[r]&V' \ar[d]\ar[r]&Z\times N' \ar[d]\\ 
\pi^{-1}(U) &U^\dag \ar[l]^-{a}\ar[r]_-{a'}\ar@{}[lu]|{\square}
&U'\ar[r]_-{g'}\ar@{}[lu]|{\square}&Z\times N, \ar@{}[lu]|{\square}
}
$$
where each square is the fiber product. 
By construction, all the upper horizontal morphisms are \'etale. 
Let $B_{V^\dag}$ (resp.~$B_{\widehat{V}}$) be the sum 
of the birational transform of $a'^*B_{U'}$ 
(resp.~$B_{\widetilde X}|_{\pi^{-1}(U)}$) and the 
exceptional locus of 
$V^\dag \to U^\dag$ (resp.~$\widehat{V}\to \pi^{-1}(U)$). 
Then, every stratum of $(V^\dag, \Supp B_{V^\dag})$ is 
smooth over $Z$. 
Since each 
irreducible component of $V^\dag$ 
is smooth over $Z$ 
and $V^\dag \to \widehat{V}$ 
is \'etale, we see that each irreducible component 
of $\widehat{V}$ is smooth over $Z$. 
By a similar argument, we see that every stratum 
of $(\widehat{V}, \Supp B_{\widehat{V}})$ is smooth over $Z$. 
This implies that $d\colon V\to \widetilde X$ satisfies (iii). 
\end{step}
We finish the proof of Lemma \ref{x-lem4.4}. 
\end{proof}

Before we start the proof of Theorem \ref{x-thm4.1}, 
we make an important remark on \cite{fujino-slc-trivial}. 

\begin{rem}[see Remark \ref{x-rem3.3}]\label{x-rem4.5}
In Theorem \ref{x-thm4.1}, 
we can write 
$$
K_X+B+\frac{1}{b}\ddiv(\varphi)=f^*D
$$ 
for some positive integer 
$b$ and a rational function $\varphi\in \Gamma (X, \mathcal K^*_X)$, 
where $\mathcal K_X$ is the sheaf of total quotient rings of 
$\mathcal O_X$ and $\mathcal K^*_X$ denotes 
the sheaf of invertible elements in $\mathcal K_X$, such that $b(K_X+B-f^*D)\sim 0$. 
In general, $b$ is larger than $b(F, B_F)$ in \cite[Section 6]{fujino-slc-trivial}. 
We take a $b$-fold cyclic cover $\pi\colon \widetilde X\to X$ associated 
to $b\Delta\sim 0$, where 
$\Delta=K_X+B-f^*D$, as in \cite[Section 6]{fujino-slc-trivial}. 
Then the general fiber of $h\colon V\to Y$ 
is not necessarily connected in \cite[Section 6]{fujino-slc-trivial}. 
Moreover, $V$ is not necessarily connected. 
This means that \cite[Proposition 6.3 (ii)]{fujino-slc-trivial} 
does not hold true since the 
natural map $\mathcal O_Y\to h_*\mathcal O_V$ is not always 
an isomorphism. 
Fortunately, the condition $h_*\mathcal O_V\simeq 
\mathcal O_Y$ is not necessary for the proof of the 
other properties of \cite[Proposition 6.3]{fujino-slc-trivial}. 
We note that 
the condition $h_*\mathcal O_V\simeq \mathcal O_Y$ is unnecessary in 
\cite[Lemma 7.3 and Theorem 8.1]{fujino-slc-trivial}. 
Hence it 
may be better to remove the condition $f_*\mathcal O_X\simeq 
\mathcal O_Y$ from (2) in Definition \ref{x-def3.2}. 

Let $b$ be the smallest positive integer such that 
$b(K_X+B-f^*D)\sim 0$. 
Then we can write 
$$
K_X+B+\frac{1}{b}\ddiv (\varphi)=f^*D. 
$$ 
As usual, we consider 
the $b$-fold cyclic cover 
$\pi\colon \widetilde X\to X$ associated to 
$b\Delta=\ddiv(\varphi^{-1})$, 
where $\Delta=K_X+B-f^*D$. 
Let $b^\sharp$ be any positive integer with $b^\sharp \geq 2$. We put 
$\varphi^\sharp=\varphi^{b^\sharp}$. Then we get 
$$
K_X+B+\frac{1}{bb^\sharp} \ddiv (\varphi ^\sharp)=f^*D. 
$$ 
Let $\pi^\sharp \colon X^\sharp\to X$ be the $bb^\sharp$-fold cyclic cover 
associated to $b b^\sharp \Delta=\ddiv\left((\varphi^\sharp)^{-1}\right)$. 
We take the $H$-invariant part of 
$\pi^\sharp \colon X^\sharp\to X$, where $H$ is the 
subgroup of the Galois 
group $\Gal(X^\sharp /X)\simeq \mathbb Z/bb^\sharp\mathbb Z$ 
of $\pi^\sharp\colon X^\sharp\to X$ 
corresponding to $b\mathbb Z /b b^\sharp \mathbb Z$. 
Then we can recover $\pi\colon \widetilde X\to X$. 
Note that $\pi^\sharp\colon X^\sharp \to X$ is decomposed 
into $b^\sharp$ components and that each 
component is isomorphic to $\pi\colon \widetilde X\to X$. 
\end{rem}

Let us prove Theorem \ref{x-thm4.1}. 

\begin{proof}[Proof of Theorem \ref{x-thm4.1}]
Here, we only explain how to modify the proof of 
\cite[Theorem 1.2]{fujino-slc-trivial} by using Lemma \ref{x-lem4.4}. 

By taking a completion as in \cite[Lemma 4.12]{fujino-slc-trivial}, 
we may further assume that $Y$ is projective. 
By Lemma \ref{x-lem4.4}, we can construct a 
commutative diagram (6.4) in 
\cite[Section 6]{fujino-slc-trivial} 
satisfying (a)--(g) such that 
$\Sigma_Y=\Sigma$ holds without 
taking birational modifications of $Y$. 
Here, we do not require the condition $\Supp M_Y\subset 
\Supp \Sigma_Y$ in (d) in \cite[Section 6]{fujino-slc-trivial} 
(see Remark \ref{x-rem3.8}). 
We also do not require the condition 
that the general fiber of $h\colon V\to Y$ is connected (see 
Remark \ref{x-rem4.5}). 
The covering arguments and \cite[Proposition 6.3]{fujino-slc-trivial} 
work without any modifications. 
We note that $Y$ is a smooth projective variety. 
In what follows, we apply the proof of 
\cite[Theorem 8.1]{fujino-slc-trivial}.   
Let $\gamma \colon Y' \to Y$ be a projective 
birational morphism from a normal variety $Y'$. 
By replacing $Y'$ with a higher model if 
necessary, we may assume that $Y'$ is smooth 
and that $\gamma ^{-1}\Sigma_Y$ is a simple 
normal crossing divisor on $Y'$. 
With \cite[Lemma 7.3]{fujino-slc-trivial}, we 
construct $\tau\colon \overline{Y}\to Y$ a 
unipotent reduction of the local monodromies around $\Sigma_{Y}$. 
Then the induced 
fibration 
over $\overline{Y}$ satisfies \cite[Proposition 6.3 (iv), (v)]{fujino-slc-trivial}.
As in the proof of 
\cite[Theorem 8.1]{fujino-slc-trivial}, we get a diagram: 
$$
\xymatrix{
\overline{Y} \ar[d]_{\tau}& \overline{Y}'\ar[l]_{\gamma'}\ar[d]^{\tau'}\\ 
Y &Y'\ar[l]^{\gamma} 
}
$$
such that $\tau'$ is finite and the induced 
fibration over $\overline{Y}'$ satisfies 
\cite[Proposition 6.3 (iv), (v)]{fujino-slc-trivial}.
By \cite[Theorem 3.1]{fujino-slc-trivial}, we see that 
$\mathbf M_{\overline{Y}}$ is a nef Cartier divisor 
and $\gamma'^{*}\mathbf M_{\overline{Y}}=\mathbf M_{\overline{Y}'}$. 
Moreover, we have $\tau^{*}\mathbf M_{Y}
=\mathbf M_{\overline{Y}}$ and 
$\tau'^{*}\mathbf M_{Y'}=\mathbf M_{\overline{Y}'}$ 
because $\tau$ and $\tau'$ are both 
finite (see \cite[Lemma 4.10]{fujino-slc-trivial}). 
Thus, we have that $\mathbf M_{Y}$ is a 
nef $\mathbb{Q}$-divisor and 
$\gamma^{*}\mathbf M_{Y}=\mathbf M_{Y'}$. This 
is Theorem \ref{x-thm4.1} (ii). 
Theorem \ref{x-thm4.1} (i) immediately follows from 
Theorem \ref{x-thm4.1} (ii). 
So we are done. 
\end{proof}

\section{Fundamental theorem for basic $\mathbb R$-slc-trivial 
fibrations}\label{x-sec5}

In this section, we will establish the following 
fundamental theorem for basic $\mathbb R$-slc-trivial 
fibrations. 

\begin{thm}[see Theorem \ref{x-thm1.8}]\label{x-thm5.1}
Let $f\colon (X, B)\to Y$ be a 
basic $\mathbb R$-slc-trivial fibration such 
that $Y$ is a smooth quasi-projective variety. 
We write 
$K_X+B\sim _{\mathbb R} f^*D$. 
Assume that 
there exists a simple normal crossing divisor 
$\Sigma$ on $Y$ such that 
$\Supp D\subset \Sigma$ and 
that every stratum of $(X, \Supp B)$ is 
smooth over $Y\setminus \Sigma$. 
Then 
\begin{itemize}
\item[(i)] $\mathbf K +\mathbf B=
\overline {\mathbf K_Y+\mathbf B_Y}$ holds, and 
\item[(ii)] $\mathbf M_Y$ is a potentially 
nef $\mathbb R$-divisor 
on $Y$ with $\mathbf M=\overline {\mathbf M_Y}$. 
\end{itemize}
\end{thm}

By Theorem \ref{x-thm5.1}, which is 
obviously a generalization of Theorem \ref{x-thm4.1}, 
we can use the theory of 
basic slc-trivial fibrations in \cite{fujino-slc-trivial} and 
\cite{fujino} for $\mathbb R$-divisors. 
The following formulation may be useful. Hence 
we state it explicitly here for the reader's convenience. 
We note that if $f\colon (X, B)\to Y$ is a 
basic $\mathbb Q$-slc-trivial 
fibration then Corollary \ref{x-cor5.2} is nothing but 
\cite[Theorem 1.2]{fujino-slc-trivial}. 

\begin{cor}[{\cite[Theorem 1.2]{fujino-slc-trivial}}]\label{x-cor5.2}
Let $f\colon (X, B)\to Y$ be a basic 
$\mathbb R$-slc-trivial fibration and 
let $\mathbf B$ and $\mathbf M$ be the 
discriminant and moduli $\mathbb R$-b-divisors associated 
to $f\colon (X, B)\to Y$, respectively. 
Then we have the following properties: 
\begin{itemize}
\item[(i)] $\mathbf K+\mathbf B$ is $\mathbb R$-b-Cartier, 
where $\mathbf K$ is the canonical 
b-divisor of $Y$, and 
\item[(ii)] $\mathbf M$ is b-potentially nef, that is,  
there exists a proper birational morphism $\sigma\colon Y'\to Y$ 
from a normal variety $Y'$ such that 
$\mathbf M_{Y'}$ is a potentially nef $\mathbb R$-divisor on $Y'$ and 
that $\mathbf M=\overline{\mathbf M_{Y'}}$ holds. 
\end{itemize}
\end{cor}

\begin{rem}[{see \cite[Corollary 1.4]{fujino-fujisawa-liu}}]\label{x-rem5.3}
In Theorem \ref{x-thm5.1} and Corollary \ref{x-cor5.2}, 
we can easily see that 
$\mathbf M_Y$ is semi-ample when $Y$ is a curve by 
Theorem \ref{x-thm3.7} and Lemma \ref{x-lem5.4} below. 
\end{rem}

Let us start with an easy lemma. 

\begin{lem}\label{x-lem5.4}
Let $f\colon (X, B)\to Y$ be a basic 
$\mathbb R$-slc-trivial fibration with 
$K_X+B\sim _\mathbb R f^*D$. 
Then there exist a $\mathbb Q$-divisor 
$B_i$ on $X$, a $\mathbb Q$-Cartier $\mathbb Q$-divisor 
$D_i$ on $Y$, and a positive 
real number $r_i$ for $1\leq i\leq k$ such that 
\begin{itemize}
\item[(1)] $\sum _{i=1}^k r_i=1$ with 
$\sum _{i=1}^k r_i B_i=B$ and 
$\sum _{i=1}^k r_i D_i=D$, 
\item[(2)] $\Supp B=\Supp B_i$, $\lfloor B^{>1}\rfloor =
\lfloor B^{>1}_i\rfloor$, and 
$\lceil -(B^{<1})\rceil=\lceil -(B^{<1}_i)\rceil$ hold 
for every $i$, 
\item[(3)] if $\coeff _S (B)\in \mathbb Q$ for a prime 
divisor $S$ on $X$, 
then $\coeff _S (B)=\coeff _S (B_i)$ holds for 
every $i$, 
\item[(4)] $\Supp D=\Supp D_i$ holds for every $i$, 
\item[(5)] if $\coeff _T(D)\in \mathbb Q$ for a 
prime divisor $T$ on $Y$, then $\coeff _T(D)= 
\coeff _T (D_i)$ holds for 
every $i$, and 
\item[(6)] $K_X+B_i\sim _{\mathbb Q} f^*D_i$ holds for every $i$. 
\end{itemize}
In particular, $f\colon (X, B_i)\to Y$ is a 
basic $\mathbb Q$-slc-trivial fibration 
with $K_X+B_i\sim 
_{\mathbb Q} f^*D_i$ for every $i$. 
Moreover, if $t_1, \ldots, t_k$ are real numbers 
such that $0\leq t_i\leq 1$ for every $i$ with $\sum _{i=1}^k t_i=1$, 
then $f\colon \left(X, \sum _{i=1}^k t_i B_i\right)\to Y$ is a 
basic $\mathbb R$-slc-trivial fibration with $K_X+\sum _{i=1}^k t_i B_i\sim 
_{\mathbb R} f^*\left(\sum _{i=1}^k t_i D_i\right)$. 
\end{lem}

\begin{proof}
The proof of \cite[Lemma 11.1]{fujino-slc-trivial} works 
with some suitable minor modifications. Therefore, 
we can take $B_i$, $D_i$, and $r_i$ for 
$1\leq i\leq k$ satisfying (1)--(6). 
By (2), $B_i=B^{\leq 1}_i$ holds over the generic 
point of $Y$ for every $i$. 
By (2) again, $\rank f_*\mathcal O_X(\lceil -(B^{<1}_i)\rceil)
=\rank f_*\mathcal O_X(\lceil -(B^{<1})\rceil)=1$. 
Hence $f\colon (X, B_i)\to Y$ is a basic $\mathbb Q$-slc-trivial 
fibration with $K_X+B_i\sim _{\mathbb Q} f^*D_i$ for every $i$. 
We put $\widetilde B=\sum _{i=1}^k t_i B_i$. 
Then $\widetilde B=\widetilde B^{\leq 1}$ holds over the 
generic point of $Y$ by (2). 
By (2) again, we see that $\lceil -(\widetilde B^{<1})\rceil 
=\lceil -(B^{<1})\rceil$ holds. 
Therefore, $f\colon (X, \widetilde B)\to Y$ is a 
basic $\mathbb R$-slc-trivial fibration. 
\end{proof}

We also need the following lemma. 

\begin{lem}\label{x-lem5.5}
Let $f\colon (X, B)\to Y$ be a basic $\mathbb R$-slc-trivial 
fibration. Let $\mathbf B$ denote the discriminant 
$\mathbb R$-b-divisor associated to $f\colon (X, B)\to Y$. 
Suppose that there are $\mathbb{Q}$-divisors 
$B_{1},\dots, B_{k}$ on $X$ and real numbers 
$r_{1},\dots, r_{k}$ such that 
$\sum _{i=1}^k r_i=1$ and $\sum _{i=1}^k r_i B_i=B$.  
We put 
$$ \mathcal P=
\left\{\left.\sum _{i=1}^{k} t_i B_i \, \right|\, \text{$0\leq t_i\leq 1$ 
for every $i$ with $\sum _{i=1}^{k}t_i=1$}\right\}.$$
Assume that $f\colon (X, \Delta)\to Y$ 
has the structure of a basic 
$\mathbb{R}$-slc-trivial fibration for 
every $\Delta\in \mathcal P$. 
For $\Delta \in \mathcal{P}$, 
$\mathbf B^{\Delta}$ denotes the 
discriminant $\mathbb{R}$-b-divisor of 
the basic $\mathbb{R}$-slc-trivial fibration $f\colon (X,\Delta)\to Y$. 
Then, we can find $\Delta_{1},\dots, 
\Delta_{l}\in \mathcal P$ which are 
$\mathbb{Q}_{\geq 0}$-linear combinations of 
$B_{1},\dots, B_{k}$ and positive real 
numbers $s_{1},\dots,s_{l}$ such that 
\begin{itemize}
\item$\sum _{j=1}^l s_j=1$ and $\sum _{j=1}^l s_j \Delta_j=B$, and 
\item $\mathbf B_{Y}=\sum_{j=1}^{l}s_{j}\mathbf B^{\Delta_{j}}_{Y}$. 
\end{itemize}
Here, $\mathbf B_{Y}$ {\em{(}}resp.~$\mathbf B^{\Delta_{j}}_{Y}${\em{)}} 
is the 
trace of the discriminant $\mathbb R$-b-divisor 
$\mathbf B$ {\em{(}}resp.~$\mathbf B^{\Delta_{j}}${\em{)}} on $Y$. 
\end{lem}

\begin{proof}
Since $\mathbf B$ is an $\mathbb R$-b-divisor, it is sufficient 
to prove the lemma for a resolution of $Y'\to Y$ and the induced basic slc-trivial 
fibrations $f'\colon (X',B_{X'})\to Y$ and $f'\colon (X',(B_{i})_{X'})\to Y'$. 
Moreover, by Definition \ref{x-def3.5} and taking the 
normalization of $X$, we may assume that $X$ is a disjoint union of smooth 
varieties. 
Therefore, by replacing $X$, $Y$, $B$, 
and $B_{i}$, 
we may assume that $Y$ is smooth and 
there are simple normal 
crossing divisors $\Sigma_{X}$ on $X$ 
and $\Sigma_{Y}$ on $Y$ such that
\begin{itemize}
\item
$\Supp B\subset \Sigma_{X}$ 
and $\Supp B_{i}\subset \Sigma_{X}$ for every $i$, 
\item
$\Sigma_{X}^{v}\subset f^{-1}
\Sigma_{Y}\subset \Sigma_{X}$, where 
$\Sigma_{X}^{v}$ is the vertical part of $\Sigma_{X}$, 
\item
$f$ is smooth over $Y\setminus \Sigma_{Y}$, and
\item
$\Sigma_{X}$ is relatively 
simple normal crossing over $Y\setminus \Sigma_{Y}$. 
\end{itemize}  
Then it is clear that $\Supp \mathbf B_{Y}\subset \Sigma_{Y}$ and $\Supp 
\mathbf B^{\Delta}_{Y} \subset \Sigma_{Y}$ for all $\Delta \in \mathcal{P}$. 
We consider a rational convex polytope 
$$\mathcal C=\left\{\boldsymbol{v}=(v_1, \ldots, v_k)\in 
[0, 1]^k\, 
\left|\, \text{$\sum _jv_j=1$}\right.\right\}\subset [0, 1]^k.$$ 
Then we may identify $\mathcal C$ with 
$\mathcal P$ by putting $\Delta_{\boldsymbol{v}}=
\sum _i v_i B_i\in \mathcal P$ for 
$\boldsymbol{v}=(v_1, \ldots, v_k)\in \mathcal C$. 
We define $\boldsymbol{v}_0 \in \mathcal C$ to 
be the point such that 
$\Delta_{\boldsymbol{v}_0}=B$. 

Fix a prime divisor $Q$ on $Y$ which is a 
component of $\Sigma_{Y}$. 
We shrink $Y$ near the generic point of 
$Q$ so that all components of $f^{*}Q$ dominate $Q$. 
We can write $f^{*}Q=\sum_{i}m_{P_{i}}P_{i}$, 
where $P_{i}$ are components of $\Sigma_{X}$ such that $f(P_{i})=Q$, and 
$m_{P_{i}}=\coeff_{P_{i}}(f^{*}Q)$. 
We fix a component $P_{(B,Q)}$ of $f^{*}Q$ such that
$$\frac{1-\coeff_{P_{(B,Q)}}(B)}{m_{P_{(B,Q)}}}=
\underset{P_{i}}\min\left\{\frac{1-
\coeff_{P_{i}}(B)}{m_{P_{i}}}\right\}.$$
Note that $\frac{1-\coeff_{P_{(B,Q)}}(B)}{m_{P_{(B,Q)}}}$ 
is the log canonical 
threshold of $(X,B)$ with respect to $f^*Q$ over the generic point 
of $Q$ because $(X,B+\mu f^{*}Q)$ is 
sub log canonical over the generic point of $Q$ if and only 
if $\coeff_{P_{i}}(B)+\mu m_{P_{i}}\leq 1$ for all $P_{i}$. 
For every component $P_{i}$ of $f^{*}Q$, 
we can define a function 
$$H^{(P_{i})}(\boldsymbol{v}):=
\frac{1-\coeff_{P_{(B,Q)}}(\Delta_{\boldsymbol{v}})}
{m_{P_{(B,Q)}}}-\frac{1-\coeff_{P_{i}}(\Delta_{\boldsymbol{v}})}{m_{P_{i}}}$$
and the half space
$$H^{(P_{i})}_{\leq0}:=\{\boldsymbol{v}
\in \mathcal{C}\;|\; H^{(P_{i})}(\boldsymbol{v})\leq 0\}.$$
It is easy to check that $H^{(P_{i})}$ are 
rational affine functions and the half spaces $H^{(P_{i})}_{\leq0}$ contain 
$\boldsymbol{v}_{0}$ since 
$\boldsymbol{v}_{0}$ is the point 
such that $\Delta_{\boldsymbol{v}_{0}}=B$. 
Therefore, the set
$$\mathcal{C}_{Q}:=\mathcal{C}\cap 
\left(\bigcap_{P_{i}}H^{(P_{i})}_{\leq0}\right)$$
is a rational polytope in $\mathcal{C}$ 
containing $\boldsymbol{v}_{0}$, where 
$P_{i}$ runs over components of $f^{*}Q$. 
We put 
\begin{equation*}
\begin{split}
t(\Delta_{\boldsymbol{v}},Q):=&1-\coeff_{Q}
\left(\mathbf B^{\Delta_{\boldsymbol{v}}}_{Y}\right)\\
=&\sup\{\mu\in \mathbb{R}\,|\,\text{$(X,\Delta_{\boldsymbol{v}}+
\mu f^{*}Q)$ is sub log canonical over the generic point of $Q$}\}.
\end{split}
\end{equation*}
Then, by the definitions of $H^{(P_{i})}_{\leq0}$, 
every $\boldsymbol{v}\in \mathcal{C}_{Q}$ satisfies
\begin{equation}\label{x-eq5.1}
t(\Delta_{\boldsymbol{v}},Q)=\underset{P_{i}}
\min\left\{\frac{1-\coeff_{P_{i}}(\Delta_{\boldsymbol{v}})}{m_{P_{i}}}\right\}
=\frac{1-\coeff_{P_{(B,Q)}}(\Delta_{\boldsymbol{v}})}{m_{P_{(B,Q)}}}.
\end{equation}
Here, to prove the first equality we used 
the fact that $(X,\Delta_{\boldsymbol{v}}+\mu f^{*}Q)$ 
is sub log canonical over the generic point 
of $Q$ if and only if 
$\coeff_{P_{i}}(\Delta_{\boldsymbol{v}})+\mu m_{P_{i}}\leq 1$ for all $P_{i}$. 

Finally, we define
$$\mathcal{C'}:=\bigcap_{Q}\mathcal{C}_{Q},$$
where $Q$ runs over all irreducible components of 
$\Sigma_{Y}$. 
It is easy to see that $\mathcal{C'}$ is a 
rational polytope in $\mathcal{C}$ and $\mathcal{C'}$ 
contains $\boldsymbol{v}_{0}$. 
Thus, we can find rational points 
$\boldsymbol{v}_{1},\ldots, \boldsymbol{v}_{l}$ 
and positive real numbers $s_{1},\ldots, s_{l}$ 
such that $\sum_{j=1}^{l}s_{j}=1$ and $\sum_{j=1}^{l}s_{j} \boldsymbol{v}_{j}
=\boldsymbol{v}_{0}$. 
We put $\Delta_{j}=\Delta_{\boldsymbol{v}_{j}}$ for each $1\leq j \leq l$. 
Then $B=\Delta_{\boldsymbol{v}_{0}}=\sum_{j=1}^{l}s_{j}\Delta_{j}$. 
For every component $Q$ of 
$\Sigma_{Y}$, the equation \eqref{x-eq5.1} 
implies that 
\begin{equation*}
\begin{array}{llll}
t(B,Q)&=&\dfrac{1-\coeff_{P_{(B,Q)}}(B)}
{m_{P_{(B,Q)}}} &\qquad\qquad(\text{see \eqref{x-eq5.1}})\\\\
&=&\dfrac{1-\coeff_{P_{(B,Q)}}
\left(\sum_{j=1}^{l}s_{j}
\Delta_{j}\right)}{m_{P_{(B,Q)}}}&\qquad
\qquad(B=\sum_{j=1}^{l}s_{j}\Delta_{j})\\\\
&=&\sum_{j=1}^{l}s_{j}\cdot
\dfrac{1-\coeff_{P_{(B,Q)}}(\Delta_{j})}
{m_{P_{(B,Q)}}}&\qquad\qquad(\sum_{j=1}^{l}s_{j}=1)\\\\
&=& \sum_{j=1}^{l}s_{j}
\cdot t(\Delta_{j},Q) &\qquad\qquad(\text{see \eqref{x-eq5.1}}).
\end{array}
\end{equation*}
Since $t(\Delta_{j},Q)=1-\coeff_{Q}
\left(\mathbf B^{\Delta_{j}}_{Y}\right)$ for 
every $1\leq j \leq l$ and every irreducible 
component $Q$ of $\Sigma_{Y}$, we see 
that $\mathbf B_{Y}=\sum_{j=1}^{l}s_{j}\mathbf B^{\Delta_{j}}_{Y}$. 
\end{proof}

We are ready to prove Theorem \ref{x-thm5.1}. 

\begin{proof}[Proof of Theorem \ref{x-thm5.1}]
Fix an arbitrary projective birational morphism 
$\sigma \colon Y' \to Y$ from a normal quasi-projective variety $Y'$, and let 
$$
\xymatrix{
   (X', B_{X'}) \ar[r]^{\mu} \ar[d]_{f'} & (X, B)\ar[d]^{f} \\
   Y' \ar[r]_{\sigma} & Y 
} 
$$
be the induced basic $\mathbb R$-slc-trivial fibration (see 
Definition \ref{x-def3.4}).
It is sufficient to show that 
$\sigma^{*}(K_{Y}+\mathbf B_Y)
=K_{Y'}+\mathbf B_{Y'}$ and $\mathbf M_Y$ is a potentially 
nef $\mathbb R$-divisor on $Y$ with $\sigma^{*}\mathbf M_{Y}=\mathbf M_{Y'}$. 

We pick $\mathbb{Q}$-divisors 
$B_{1},\dots, B_{k}$ on $X$, 
$\mathbb{Q}$-divisors $D_{1},\dots, D_{k}$ on $Y$ and 
positive real numbers $r_{1},\dots, r_{k}$ as in Lemma \ref{x-lem5.4}. 
Then, the following properties hold.
\begin{itemize}
\item$\sum _{i=1}^k r_i=1$ with 
$\sum _{i=1}^k r_i B_i=B$ and 
$\sum _{i=1}^k r_i D_i=D$, 
\item $\Supp B=\Supp B_i$ and $\Supp D=\Supp D_i$ hold for every $i$, and
\item $K_X+B_i\sim _{\mathbb Q} f^*D_i$ holds for every $i$. 
\end{itemize}
We put $D'_{i}=\sigma^{*}D_{i}$ and 
we define $B'_{i}$ by $K_{X'}+B'_{i}=\mu^{*}(K_{X}+B_{i})$ for any $1\leq i\leq k$. 
Then $f'\colon(X',B'_{i})\to Y'$ are basic 
$\mathbb Q$-slc-trivial fibrations with $K_{X'}+B'_{i}\sim_{\mathbb{Q}}f'^{*}D'_{i}$. 
As in Lemma \ref{x-lem5.5}, we put 
$$
\mathcal P'=\left\{\left.\sum _{i=1}^{k} t_i B'_i \, 
\right|\, 
\text{$0\leq t_i\leq 1$ for every $i$ with $\sum _{i=1}^{k}t_i=1$}\right\}.
$$
We may assume that $f'\colon (X', \Delta)\to Y'$ is a basic 
$\mathbb R$-slc-trivial fibration for every $\Delta\in \mathcal P'$. 
We define $\mathcal P'_{\mathbb{Q}}$ by
$$
\mathcal P'_{\mathbb{Q}}:=\left\{\left.\sum _{i=1}^{k} t_i B'_i \, 
\right|\, 
\text{$t_i\in \mathbb Q$ and 
$0\leq t_i\leq 1$ for every $i$ with $\sum _{i=1}^{k}t_i=1$}\right\}.
$$
Note that $B_{X'}\in \mathcal P'$. 

Pick any $\Delta =\sum_{i=1}^{k} t_i B'_i\in \mathcal P'_{\mathbb{Q}}$.  
Since $\mu_{*}B'_{i}=B_{i}$, we 
have $\mu_{*}\Delta=\sum_{i=1}^{k} t_i B_i$ 
such that  $t_{i}\in \mathbb{Q}$. 
Therefore, the morphism $f\colon (X,\mu_{*}\Delta)\to Y$ 
is a basic $\mathbb Q$-slc-trivial fibration 
such that $K_{X}+\mu_{*}\Delta \sim_{\mathbb{Q}}f^{*}
\bigl(\sum_{i=1}^{k} t_i D_{i}\bigr)$. 
Let $\mathbf B^{\Delta}$ and $\mathbf M^{\Delta}$ 
be the discriminant $\mathbb Q$-b-divisor and the moduli 
$\mathbb Q$-b-divisor of the basic $\mathbb Q$-slc-trivial 
fibration $f\colon(X,\mu_{*}\Delta)\to Y$, respectively. 
Because we have $\Supp
\bigl(\sum_{i=1}^{k} t_i D_{i}\bigr)\subset \Supp D$ and $\Supp 
\mu_{*}\Delta\subset \Supp B$, we may apply Theorem \ref{x-thm4.1}. 
Therefore, for every $\Delta \in \mathcal P'_{\mathbb{Q}}$ 
it follows that $\sigma^{*}(K_{Y}+\mathbf B^{\Delta}_{Y})
=K_{Y'}+\mathbf B^{\Delta}_{Y'}$ and $\mathbf M^{\Delta}_Y$ 
is a potentially nef $\mathbb Q$-divisor on $Y$ with 
$\sigma^{*}\mathbf M^{\Delta}_{Y}=\mathbf M^{\Delta}_{Y'}$. 
It also follows from the construction that $f'\colon (X',\Delta)\to Y'$ 
is the basic $\mathbb Q$-slc-trivial fibration induced from 
$f\colon (X,\mu_{*}\Delta)\to Y$ 
such that 
$K_{X'}+\Delta\sim_{\mathbb{Q}}f'^{*}\bigl(\sum_{i=1}^{k} t_i D'_{i}\bigr)$. 
It is because $K_{X'}+\Delta=\mu^{*}(K_{X}+\mu_{*}\Delta)$ by construction. 

We apply Lemma \ref{x-lem5.5} to $f'\colon(X',B_{X'})\to Y'$ 
and $\mathcal{P}'$. 
Then, we can find 
$\Delta_{1},\dots, \Delta_{l}\in \mathcal P'_{\mathbb{Q}}$ 
and positive real numbers $s_{1},\dots,s_{l}$ such that 
\begin{itemize}
\item$\sum _{j=1}^l s_j=1$ and $\sum _{j=1}^l s_j \Delta_j=B_{X'}$, and 
\item $\mathbf B_{Y'}=\sum_{j=1}^{l}s_{j}\mathbf B^{\Delta_{j}}_{Y'}$. 
\end{itemize}
Since $\mathbf B$ and $\mathbf B^{\Delta_{j}}$ 
are $\mathbb R$-b-divisors, we have $\mathbf B_{Y}=
\sum_{j=1}^{l}s_{j}\mathbf B^{\Delta_{j}}_{Y}$. 
Then 
\begin{equation*}
\begin{split}
\sigma^{*}(K_{Y}+\mathbf B_Y)&=
\sigma^{*}\left(K_{Y}+\sum_{j=1}^{l}s_{j}\mathbf B^{\Delta_{j}}_{Y}\right)
=\sum_{j=1}^{l}s_{j}\sigma^{*}(K_{Y}+\mathbf B^{\Delta_{j}}_{Y})\\
&=\sum_{j=1}^{l}s_{j}(K_{Y'}+\mathbf B^{\Delta_{j}}_{Y'})
=K_{Y'}+\sum_{j=1}^{l}s_{j}\mathbf B^{\Delta_{j}}_{Y'}\\
&=K_{Y'}+\mathbf B_{Y'}.
\end{split}
\end{equation*}
Therefore, we have $\sigma^{*}(K_{Y}+\mathbf B_Y)
=K_{Y'}+\mathbf B_{Y'}$. Recalling 
that $\sigma \colon Y' \to Y$ is an arbitrary 
projective birational morphism, we see 
that (i) of Theorem \ref{x-thm5.1} holds, i.e.,
$$\mathbf K +\mathbf B=
\overline {\mathbf K_Y+\mathbf B_Y}.$$ 
As in the third paragraph, for each $j$ we 
define $D'_{\Delta_{j}}$ to be the $\mathbb{Q}$-divisor 
on $Y'$ associated to the basic $\mathbb Q$-slc-trivial 
fibration $f'\colon(X',\Delta_{j})\to Y'$. 
Note that $K_{X'}+\Delta_{j}
\sim_{\mathbb{Q}}f'^{*}D'_{\Delta_{j}}$ for all $j$. 
Since $\sum _{j=1}^l s_j=1$ and 
$\sum _{j=1}^l s_j \Delta_j=B_{X'}$, we 
have $\sigma^{*}D=\sum_{j=1}^{l}s_{j}D'_{\Delta_{j}}$. 
By the relation $\mathbf B_{Y'}=
\sum_{j=1}^{l}s_{j}\mathbf B^{\Delta_{j}}_{Y'}$ 
and the definition of the moduli $\mathbb R$-b-divisors 
(see Definition \ref{x-def3.5}), we have
$$\mathbf M_{Y'}=\sum_{j=1}^{l}s_{j}
\mathbf M^{\Delta_{j}}_{Y'}\qquad \text{and} \qquad 
\mathbf M_{Y}=\sum_{j=1}^{l}s_{j}\mathbf M^{\Delta_{j}}_{Y}$$
Then $\mathbf M_Y$ is a potentially nef $\mathbb R$-divisor on $Y$ and 
\begin{equation*}
\begin{split}
\sigma^{*}\mathbf M_{Y}&=
\sigma^{*}\left(\sum_{j=1}^{l}s_{j}
\mathbf M^{\Delta_{j}}_{Y}\right)=\sum_{j=1}^{l}s_{j}
\mathbf M^{\Delta_{j}}_{Y'}=\mathbf M_{Y'}.
\end{split}
\end{equation*}
Here, we used 
$\sigma^{*}\mathbf M^{\Delta_{j}}_{Y}
=\mathbf M^{\Delta_{j}}_{Y'}$ for every $j$, which follows 
from the third paragraph. 
We complete the proof. 
\end{proof}

The following result is essentially obtained in the proof of 
Theorem \ref{x-thm5.1}. 
We explicitly state it here for future use. 

\begin{thm}\label{x-thm5.6}
Let $f\colon (X, B)\to Y$ be a 
basic $\mathbb R$-slc-trivial 
fibration with $K_X+B\sim _{\mathbb R} f^*D$. 
Then there are $\mathbb{Q}$-divisors $B_{1},\dots, B_{l}$ 
on $X$, $\mathbb Q$-Cartier  
$\mathbb{Q}$-divisors $D_{1},\dots, D_{l}$ on 
$Y$ and positive real numbers $r_{1},\dots, r_{l}$ satisfying the 
following properties.
\begin{itemize}
\item$\sum _{j=1}^l r_j=1$ with 
$\sum _{j=1}^l r_j B_j=B$ and 
$\sum _{j=1}^l r_j D_j=D$, 
\item $\Supp B=\Supp B_j$, $\lfloor B^{>1}\rfloor =
\lfloor B^{>1}_j\rfloor$, and 
$\lceil -(B^{<1})\rceil=\lceil -(B^{<1}_j)\rceil$ hold 
for every $j$, 
\item if $\coeff _S (B)\in \mathbb Q$ for a prime 
divisor $S$ on $X$, 
then $\coeff _S (B)=\coeff _S (B_j)$ holds for 
every $j$, 
\item $\Supp D=\Supp D_j$ holds for every $j$, 
\item if $\coeff _T(D)\in \mathbb Q$ for a 
prime divisor $T$ on $Y$, then $\coeff _T(D)= 
\coeff _T (D_j)$ holds for 
every $j$, 
\item $K_X+B_j\sim _{\mathbb Q} f^*D_j$ holds for every $j$, 
\item $\mathbf B=\sum _{j=1}^l r_j \mathbf B_j$ as b-divisors, 
where $\mathbf B$ {\em{(}}resp.~$\mathbf B_j${\em{)}} is 
the discriminant $\mathbb{R}$-b-divisor {\em{(}}resp.~the 
discriminant $\mathbb{Q}$-b-divisor{\em{)}} of 
$f\colon (X,B)\to Y$ {\em{(}}resp.~$f\colon (X,B_{j})\to Y${\em{)}}, and 
\item $\mathbf M=\sum _{j=1}^l r_i \mathbf M_j$ as b-divisors, where 
$\mathbf M$ {\em{(}}resp.~$\mathbf M_j${\em{)}} is 
the moduli $\mathbb{R}$-b-divisor {\em{(}}the moduli 
$\mathbb{Q}$-b-divisor{\em{)}} associated to $f\colon (X,B)\to Y$ 
{\em{(}}resp.~$f\colon (X,B_{j})\to Y${\em{)}}. 
\end{itemize}
\end{thm}

\begin{proof}[Sketch of Proof]
It can be proved by Theorem \ref{x-thm5.1}, Lemma \ref{x-lem5.4} 
and Lemma \ref{x-lem5.5}. 
We only outline the proof.  

We note that the properties of 
Theorem \ref{x-thm5.6} except the last two 
properties correspond to (1)--(6) of Lemma \ref{x-lem5.4} respectively.
By Lemma \ref{x-lem5.4}, we can find $\mathbb{Q}$-divisors 
$\widetilde{B}_{1},\dots, \widetilde{B}_{k}$ on 
$X$, $\mathbb{Q}$-Cartier $\mathbb Q$-divisors 
$\widetilde{D}_{1},\dots, \widetilde{D}_{k}$ 
on $Y$ and positive real numbers 
$\widetilde{r}_{1},\dots, \widetilde{r}_{k}$ 
satisfying (1)--(6) of Lemma \ref{x-lem5.4}. 
Then $\widetilde{B}_{i}$, $\widetilde{D}_{i}$, 
and $\widetilde{r}_{i}$ satisfy all the properties 
of Theorem \ref{x-thm5.6} except the last two properties. 
More specifically, $\widetilde{B}_{i}$, $\widetilde{D}_{i}$, 
and $\widetilde{r}_{i}$ satisfy 
\begin{itemize}
\item$\sum _{i=1}^k \widetilde{r}_i=1$ with 
$\sum _{i=1}^k \widetilde{r}_i \widetilde{B}_i=B$ and 
$\sum _{i=1}^k \widetilde{r}_i \widetilde{D}_i
=D$ (see (1) of Lemma \ref{x-lem5.4}), 
\item $\Supp B=\Supp \widetilde{B}_i$ and 
$\Supp D=\Supp \widetilde{D}_i$ hold for every $i$, 
\item $K_X+\widetilde{B}_i\sim _{\mathbb Q} f^*\widetilde{D}_i$ 
holds for every $i$  (see (6) of Lemma \ref{x-lem5.4}), 
\end{itemize} 
and (2)--(5) in Lemma \ref{x-lem5.4}. 
We take a smooth higher model $\sigma \colon Y'\to Y$ 
so that the induced basic $\mathbb{R}$-slc-trivial 
fibration $f'\colon (X',B')\to Y'$ satisfies the property 
that there exists a simple normal crossing divisor 
$\Sigma'$ on $Y'$ such that 
$\Supp \sigma^{*}D\subset \Sigma'$ and 
that every stratum of $(X', \Supp B')$ is 
smooth over $Y'\setminus \Sigma'$. 
The morphism $X'\to X$ is denoted by $\mu$. 
For each $1\leq i\leq k$, 
let $\widetilde{B}'_i$ be a $\mathbb{Q}$-divisor on 
$X'$ defined by $K_{X'}+\widetilde{B}'_i=\mu^{*}(K_{X}+\widetilde{B}_i)$. 
Note that 
$K_{X'}+\widetilde{B}'_i\sim_{\mathbb{Q}}f'^{*}\sigma^{*}\widetilde{D}_{i}$. 
We may assume that $\Supp \sigma^{*}\widetilde{D}_{i}\subset \Sigma'$ and 
that every stratum of $(X', \Supp \widetilde{B}'_i)$ is 
smooth over $Y'\setminus \Sigma'$ for every $i$ by taking 
$\sigma\colon Y'\to Y$ suitably. 
We define 
$$ \mathcal P=
\left\{\left.\sum _{i=1}^{k} t_i 
\widetilde{B}'_i \, \right|\, \text{$0\leq t_i\leq 1$ 
for every $i$ with $\sum _{i=1}^{k}t_i=1$}\right\}.$$
By Lemma \ref{x-lem5.5}, we can find 
$B'_{1},\dots, B'_{l}\in \mathcal P$ which 
are $\mathbb{Q}_{\geq 0}$-linear combinations of 
$\widetilde{B}'_{1},\dots, \widetilde{B}'_{l}$ 
and positive real numbers $r_{1},\dots,r_{l}$ such that 
\begin{itemize}
\item$\sum _{j=1}^l r_j=1$ and $\sum _{j=1}^l r_j B'_j=B'$, and 
\item $\mathbf B_{Y'}=\sum_{j=1}^{l}r_{j}\mathbf B_{jY'}$. 
\end{itemize}
Here, $\mathbf B_j$ is the 
discriminant $\mathbb{Q}$-b-divisor associated 
to $f'\colon (X',B'_{j})\to Y'$. 
By Theorem \ref{x-thm5.1}, we have $\mathbf K+\mathbf B=
\overline{\mathbf K_{Y'}+\mathbf B_{Y'}}$ and $\mathbf K+
\mathbf B_j=\overline{\mathbf K_{Y'}+\mathbf B_{jY'}}$. 
We put $B_j=\mu_{*}B'_j$ for each $1\leq j \leq l$. 
Then we can find $\mathbb{Q}$-divisors $D_{1},\dots, D_{l}$ on $Y$ 
such that $K_{X}+B_{j}\sim_{\mathbb{Q}}f^{*}D_{j}$ and 
$\sum_{j=1}^{l}r_{j}D_{j}=D$. 
By construction, we can easily see that $B_1\ldots, B_l$, $D_1, 
\ldots, D_l$, and $r_1, \ldots, r_l$ constructed above satisfy 
the desired properties. 
\end{proof}

\section{Proof of Theorem \ref{x-thm1.2}}\label{x-sec6} 

In this section, we will prove 
Theorem \ref{x-thm1.2}, which is the main 
result of this paper. 
Then we will treat Theorem \ref{x-thm1.1} 
and Corollary \ref{x-cor1.4}. 
We note that we will freely use the framework 
of quasi-log schemes 
in the proof of Theorem \ref{x-thm1.2}. 
For the details of quasi-log schemes, 
see \cite[Chapter 6]{fujino-foundations}. 
Let us start with the proof of Theorem \ref{x-thm1.2}. 

\begin{proof}[Proof of Theorem \ref{x-thm1.2}]
From Step \ref{x-step1.3.1} to Step \ref{x-step1.3.3}, 
we will define a natural quasi-log scheme structure on $Z$. 
This part is essentially contained 
in \cite[Chapter 6]{fujino-foundations} 
and \cite{fujino}. 

\setcounter{step}{0}
\begin{step}\label{x-step1.3.1}
In this step, we will give a natural quasi-log scheme 
structure on $W':=W\cup \Nlc(X, \Delta)$. 
This step is essentially the adjunction for quasi-log schemes 
(see \cite[Theorem 6.3.5 (i)]{fujino-foundations}). 

\medskip 

We put $W':=W\cup \Nlc(X, \Delta)$ 
as above. We will sketch how to 
define a natural quasi-log scheme structure on $W'$. 
Let $f\colon Y\to X$ be a projective 
birational morphism from a 
smooth quasi-projective variety $Y$ such that 
$K_Y+\Delta_Y=f^*(K_X+\Delta)$ and 
that $\Supp \Delta_Y$ is a simple normal 
crossing divisor on $Y$. 
By taking some more blow-ups, 
we may assume that 
the union of all log canonical 
centers of $(Y, \Delta_Y)$ mapped to 
$W'$ by $f$, which is denoted by $V'$, is a 
union of some irreducible 
components of $\Delta^{=1}_Y$. 
As usual, we put $A=\lceil -(\Delta^{<1}_Y)\rceil$ and 
$N=\lfloor \Delta^{>1}_Y\rfloor$ and consider 
the following short exact sequence: 
$$
0\to \mathcal O_Y(A-N-V')\to 
\mathcal O_Y(A-N)\to \mathcal O_{V'}(A-N)\to 0. 
$$
By taking $R^if_*$, we obtain: 
\begin{equation*}
\begin{split}
0&\longrightarrow f_*\mathcal O_Y(A-N-V')\longrightarrow 
f_*\mathcal O_Y(A-N)\longrightarrow 
f_*\mathcal O_{V'}(A-N)\\
&\overset{\delta}{\longrightarrow} 
R^1f_*\mathcal O_Y(A-N-V')\longrightarrow \cdots 
\end{split}
\end{equation*}
The connecting homomorphism $\delta$ is zero 
since no associated prime of 
$R^1f_*\mathcal O_Y(A-N-V')$ is contained 
in $W'=f(V')$ 
(see \cite[theorem 6.3 (i)]{fujino-fundamental} and 
\cite[Theorem 5.6.2 (i)]{fujino-foundations}). 
Hence we have: 
$$
0\to f_*\mathcal O_Y(A-N-V')\to 
f_*\mathcal O_Y(A-N)\to 
f_*\mathcal O_{V'}(A-N)\to 0. 
$$
Note that $\mathcal J_{\NLC}(X, \Delta)=f_*\mathcal O_Y(A-N)$ 
by definition. 
We put $\mathcal I_{W'}=f_*\mathcal O_Y(A-N-V')$ and 
$\mathcal I_{W'_{-\infty}}=f_*\mathcal O_{V'}(A-N)$. 
We define $\Delta_{V'}$ by $(K_Y+\Delta_Y)|_{V'}=K_{V'}+
\Delta_{V'}$. 
Then 
$$
\left(W', (K_X+\Delta)|_{W'}, f\colon (V', \Delta_{V'})\to 
W'\right)
$$ 
is a quasi-log scheme. 
By construction, $\Nqlc(W', (K_X+\Delta)|_{W'})=\Nlc(X, \Delta)$ holds. 
By construction again, a subset $C\subset X$ is a qlc stratum of 
$[W', (K_X+\Delta)|_{W'}]$ if and only if 
$C$ is a log canonical center of $(X, \Delta)$ 
included in $W$. 
We note that the above construction is independent of 
the choice of $f\colon Y\to X$ by 
\cite[Proposition 6.3.1]{fujino-foundations}. 
\end{step}
\begin{step}\label{x-step1.3.2} 
In this step, we will give a natural quasi-log scheme structure 
on $[W, (K_X+\Delta)|_W]$. 
This step is essentially \cite[Lemma 4.19]{fujino}. 

\medskip 

In Step \ref{x-step1.3.1}, 
we may further assume that the union 
of all strata of $(V', \Delta_{V'})$ 
mapped to $W\cap \Nlc(X, \Delta)$ is also a 
union of some irreducible components of $V'$. 
Let $\widehat V$ be the union of the irreducible components of $V'$ 
mapped to $W$ by $f$. 
We put $\Delta_{\widehat V}$ by $(K_Y+\Delta_Y)|_{\widehat V}
=K_{\widehat V}+\Delta_{\widehat V}$. 
Then, by the proof of \cite[Lemmas 4.18 and 4.19]{fujino},  
$$
\left(W, (K_X+\Delta)|_W, f\colon (\widehat V, \Delta_{\widehat V})
\to W\right)
$$ 
is a quasi-log scheme. 
By \cite[Lemma 4.19]{fujino}, we obtain that 
$\mathcal I_{W_{-\infty}}=\mathcal I_{W'_{-\infty}}$ 
holds and that a subset $C\subset X$ is a qlc stratum of $[W', (K_X+\Delta)|_{W'}]$ 
if and only if $C$ is a qlc stratum of $[W, (K_X+\Delta)|_W]$. 
Hence $W\cap \Nlc(X, \Delta)=W_{-\infty}$ and 
$$W\cap \left(\Nlc(X, \Delta)\cup 
\bigcup _{W\not\subset W^\dag}W^\dag\right)
=\Nqklt(W, (K_X+\Delta)|_W)$$ hold 
set theoretically, where $W^\dag$ runs over 
log canonical centers of $(X, \Delta)$ which do not contain 
$W$. 
\end{step}
\begin{step}\label{x-step1.3.3}
In this step, we will give a natural quasi-log scheme structure 
on $Z$. 
This step is nothing but \cite[Theorem 1.9]{fujino}. 

\medskip 

In Step \ref{x-step1.3.2}, we may further assume that 
the union of all strata of 
$(\widehat V, \Delta_{\widehat V})$ mapped to 
$\Nqklt(W, (K_X+\Delta)|_W)$ is a 
union of some irreducible components of $\widehat V$. 
Let $V$ be the union of the irreducible components of $\widehat V$
which are dominant onto $W$. 
Then, by the proof of \cite[Theorem 1.9]{fujino}, 
$f\colon V\to W$ factors through $Z$ and 
$$
\left(Z, \nu^*(K_X+\Delta), f\colon (V, \Delta_V)\to Z\right)
$$ 
becomes a quasi-log scheme, 
where $\Delta_V$ is defined by $(K_Y+\Delta_Y)|_V
=K_V+\Delta_V$. By construction, 
we have $\nu_*\mathcal 
I_{\Nqklt(Z, \nu^*(K_X+\Delta))}=
\mathcal I_{\Nqklt(W, (K_X+\Delta)|_W)}$. 
Hence $$\Nqklt(Z, \nu^*(K_X+\Delta))=\nu^{-1}\Nqklt(W, 
(K_X+\Delta)|_W)$$ 
holds.  
\end{step}
\begin{step}\label{x-step1.3.4}
Then $f\colon (V, \Delta_V)\to Z$ is a 
basic $\mathbb R$-slc-trivial 
fibration. Hence we can apply Corollary \ref{x-cor5.2} 
and Remark \ref{x-rem5.3} 
to $f\colon (V, \Delta_V)\to Z$. 
We note that $f\colon (V, \Delta_V)\to Z$ is a basic 
$\mathbb Q$-slc-trivial 
fibration when $K_X+\Delta$ is $\mathbb Q$-Cartier. 
In that case, Theorem \ref{x-thm3.6} with 
Theorem \ref{x-thm3.7} is sufficient. 
\end{step}
\begin{step}\label{x-step1.3.5}
By \cite[Theorem 7.1]{fujino} and 
Steps \ref{x-step1.3.1}, \ref{x-step1.3.2}, and \ref{x-step1.3.3} in its proof, 
we can construct a projective birational morphism 
$p\colon Z'\to Z$ from a 
smooth quasi-projective variety $Z'$ 
satisfying (i), (ii), (iii), and (v). 
We note that we can directly 
apply Step 3 in the proof of \cite[Theorem 7.1]{fujino} 
to basic $\mathbb R$-slc-trivial fibrations 
by Corollary \ref{x-cor5.2}. 
We also note that $\mathbf B$ is a 
well-defined $\mathbb R$-b-divisor on 
$Z$ and is independent of $f\colon Y\to X$ (see 
\cite[Lemma 5.1]{fujino-hashizume} 
and \cite[Theorem 1.2]{fujino-hashizume-proc}). 
\end{step}
\begin{step}[{see \cite[Theorem 5.4]{fujino-hashizume}}]\label{x-step1.3.6}
In this final step, we will prove (iv). 
This step is essentially \cite[Theorem 5.4]{fujino-hashizume}. 
We explain it here for the reader's convenience.  

Without 
loss of generality, we may assume that $X$ is affine 
by taking a finite affine open cover of $X$. Let 
$g_{dlt}\colon X_{dlt}\to X$ be a good dlt blow-up of 
$(X, \Delta)$ such that $K_{X_{dlt}}+\Delta_{X_{dlt}}=
g^*_{dlt}(K_X+\Delta)$ (see \cite[Lemma 3.5]{fujino-hashizume}). 
We may assume that there is an irreducible component $S$ 
of $\Delta^{=1}_{X_{dlt}}$ with $g_{dlt}(S)=W$. We put 
$$D=\Delta^{\geq 1}_{X_{dlt}}-\Supp \Delta^{\geq 1}_{X_{dlt}}=
\Delta^{> 1}_{X_{dlt}}-\Supp \Delta^{>1}_{X_{dlt}}. $$ Then $-D$ 
is semi-ample over $X$ and $\Supp D=\Nlc(X_{dlt}, \Delta_{X_{dlt}})$ 
holds set theoretically (see \cite[Lemma 3.5]{fujino-hashizume}). 
By taking the contraction morphism $\varphi\colon X_{dlt}\to X_{lc}$ 
associated to $-D$ over $X$, we get a log canonical 
modification $g_{lc}\colon X_{lc}\to X$ with $K_{X_{lc}}+
\Delta_{X_{lc}}=g^*_{lc}(K_X+\Delta)$ 
(see \cite[Theorem 1.3]{fujino-hashizume}). 
$$\xymatrix{X_{dlt} \ar[dr]_-{g_{dlt}}\ar[rr]^-\varphi& & X_{lc}
\ar[dl]^-{g_{lc}}\\& X &}$$We put $D'=\varphi_*D$. 
Then $-D'$ is ample over $X$, and $$g^{-1}_{lc}\Nlc(X, \Delta)
=\Nlc(X_{lc}, \Delta_{X_{lc}})=\Supp D'$$ holds set theoretically. 
We note that $$\Nlc(X_{dlt}, \Delta_{X_{dlt}})
=\varphi^{-1}\Nlc(X_{lc}, \Delta_{X_{lc}})=g^{-1}_{dlt}\Nlc(X, \Delta)$$ 
holds set theoretically. 
Let $\tilde{W}$ be the strict transform of $W$ on $X_{lc}$. 
Let $\tilde \nu \colon \tilde Z \to \tilde W$ be the normalization. 
Then we can easily see that $$\Supp \mathbf B^{>1}_{\tilde Z}
=\widetilde \nu ^*D'=\tilde{\nu}^{-1}\left(\Nlc (X_{lc}, \Delta_{X_{lc}})
\cap\tilde W\right) = (g_{lc}\circ 
\tilde \nu)^{-1}\left(\Nlc(X, \Delta)\cap W\right)$$ 
holds set theoretically. We note that $\mathbf B^{>1}=0$ 
over $X\setminus \Nlc(X, \Delta)$ by construction. Hence 
we obtain $\nu\circ p(\mathbf B^{>1}_{Z'})
=W\cap \Nlc(X, \Delta)$ set theoretically. \end{step}
We finish the proof of Theorem \ref{x-thm1.2}. 
\end{proof}

Finally, we prove Theorem \ref{x-thm1.1} and 
Corollary \ref{x-cor1.4}. 

\begin{proof}[Proof of Theorem \ref{x-thm1.1}]
Here, we use the same notation as in Theorem \ref{x-thm1.2}. 
We put $B_Z=\mathbf B_Z$ and $M_Z=\mathbf M_Z$ in 
Theorem \ref{x-thm1.2}. 
We note that $\mathbf M_{Z'}$ is a finite $\mathbb R_{>0}$-linear 
combination of 
potentially nef Cartier divisors 
on $Z'$ with $p_*\mathbf M_{Z'}=M_Z$. 
Hence the desired statement follows from Theorem \ref{x-thm1.2}. 
\end{proof}

\begin{proof}[Proof of Corollary \ref{x-cor1.4}]
By the definition of $\mathbf B$ in Theorem \ref{x-thm1.2} 
(see the proof of Theorem \ref{x-thm1.2} and Definition \ref{x-def1.3}), 
we can easily check that 
$B_Z$ is nothing but Shokurov's different (see 
\cite[Section 14]{fujino-fundamental}) and 
$\nu^*(K_X+\Delta)=K_Z+B_Z$ holds, where $\nu\colon Z\to W$ 
is the normalization of $W$. 
In particular, we have $M_Z=0$. 
By (A) in Theorem \ref{x-thm1.1}, 
we obtain that $(X, \Delta)$ is log canonical in a neighborhood 
of $W$ if and only if $(Z, B_Z)$ is log canonical in the usual sense. 
It recovers Kawakita's inversion of adjunction 
(see \cite[Theorem]{kawakita}). 
By (B), we see that $(Z, B_Z)$ is kawamata log terminal if and only if 
$(X, \Delta)$ is log canonical in a neighborhood of $W$ and $W$ is a 
minimal log canonical center of $(X, \Delta)$ 
(see \cite[Theorem 9.1]{fujino-fundamental} and 
\cite[Theorem 6.3.11]{fujino-foundations}). 
Note that 
$(X, \Delta)$ is purely log terminal in a neighborhood of $W$ 
if and only if $(X, \Delta)$ is log canonical 
in a neighborhood of $W$ and $W$ is a 
minimal log canonical center of $(X, \Delta)$. 
\end{proof}

We close this section with the following remark 
which summarizes the construction of the $\mathbb{R}$-b-divisors 
$\mathbf B$ and $\mathbf M$ on $Z$. 

\begin{rem}\label{x-rem6.1}
Let $X$ be a normal 
variety and let $\Delta$ be an effective $\mathbb R$-divisor 
on $X$ such that 
$K_X+\Delta$ is $\mathbb R$-Cartier. 
Let $W$ be a log canonical center of $(X, \Delta)$ 
and let $\nu\colon Z\to W$ be the normalization of $W$. 

We take a log resolution 
$f\colon Y \to X$ of $(X,\Delta)$ which is a sufficiently high birational model. 
We define $\Delta_{Y}$ by $K_{Y}+\Delta_{Y}
=f^{*}(K_{X}+\Delta)$, and let $V$ be the union 
of the irreducible components of $\Delta_{Y}^{=1}$ 
which map onto $W$.  
Let $\Delta_{V}$ be an $\mathbb{R}$-divisor on 
$V$ defined by $K_{V}+\Delta_{V}=(K_{Y}+\Delta_{Y})|_{V}$, 
then we get the morphism $f\colon (V,\Delta_{V})\to Z$ which 
has the structure of a basic $\mathbb{R}$-slc-trivial fibration. 
Then $\mathbf B$ and $\mathbf M$ are defined to be the 
discriminant $\mathbb{R}$-b-divisor and 
the moduli $\mathbb{R}$-b-divisor as in 
Definition \ref{x-def3.5}. 
By construction, we can easily check that the construction 
in the proof of Theorem \ref{x-thm1.2} and the one in Definition 
\ref{x-def1.3} define the same $\mathbb{R}$-b-divisor $\mathbf B$ 
on $Z$ (see \cite[Lemma 5.1]{fujino-hashizume} 
and \cite[Theorem 1.2]{fujino-hashizume-proc}). 
Precisely speaking, when $\dim W\leq \dim X-2$, 
we consider the $\mathbb R$-line bundle 
$\mathcal L$ on $X$ associated to $K_X+\Delta$. 
We fix an $\mathbb R$-Cartier $\mathbb R$-divisor $D$ on $Z$ 
whose associated $\mathbb R$-line bundle is the pullback 
of $\mathcal L$. 
Then we put $\mathbf M=\overline D-\mathbf K-\mathbf B$, 
where $\overline D$ is the $\mathbb R$-Cartier closure of $D$ and 
$\mathbf K$ is the canonical b-divisor of $Z$. 
\end{rem}

\section{Adjunction for codimension two log 
canonical centers}\label{x-sec7}

In this final section, we first discuss basic slc-trivial fibrations 
under some extra assumption and 
then prove adjunction for codimension two log canonical centers. 

\begin{thm}\label{x-thm7.1}
Let $f\colon (X, B)\to Y$ be a basic 
$\mathbb R$-slc-trivial fibration. 
Assume that there exists a stratum $S$ of $(X, B)$ such 
that the induced morphism $S \to Y$ is generically finite and surjective. 
Then there exists a proper birational morphism 
$p\colon Y'\to Y$ from a smooth 
quasi-projective variety $Y'$ such that $\mathbf M=
\overline {\mathbf M_{Y'}}$ with $\mathbf M_{Y'}\sim _{\mathbb R} 0$. 
In particular, $\mathbf M$ is b-semi-ample.
\end{thm}
  
\begin{proof}
By Theorem \ref{x-thm5.6}, we may assume 
that $f\colon (X, B)\to Y$ is a basic $\mathbb Q$-slc-trivial fibration. 
Let $\nu\colon X^\nu\to X$ be the normalization. 
We define a $\mathbb Q$-divisor $B^\nu$ on $X^\nu$ by 
$K_{X^\nu}+B^\nu=\nu^*(K_X+B)$. 
Note that after the reduction we may find a log canonical 
center $S$ of $(X^{\nu},B^{\nu})$ such that the 
induced morphism $S \to Y$ is generically finite and surjective. 
By \cite[Lemma 4.12]{fujino-slc-trivial}, we 
may further assume that $Y$ is a complete variety. 
By replacing $Y$ with a smooth 
higher birational model and $f\colon (X, B)\to Y$ 
with the induced basic $\mathbb Q$-slc-trivial fibration, we may 
assume that $Y$ is a smooth projective variety, 
$\mathbf M=\overline{{\mathbf M}_{Y}}$, and ${\mathbf M}_{Y}$ is nef. 
The induced morphism $S \to Y$ is denoted by $f_{S}$. 
We define a $\mathbb{Q}$-divisor $B_S$ on 
$S$ by $K_{S}+B_{S}=(K_{X^{\nu}}+B^{\nu})|_{S}$. 

From now on, we will show 
that $-{\mathbf M}_{Y}$ is $\mathbb{Q}$-linearly equivalent to an 
effective $\mathbb{Q}$-divisor. 
We consider the divisor 
$\nu^{*}f^{*}{\mathbf M}_{Y}
\sim_{\mathbb{Q}}K_{X^{\nu}}+B^{\nu}-\nu^{*}f^{*}(K_{Y}+\mathbf{B}_{Y})$. 
By restricting it to $S$, we get the relation 
$f_{S}^{*}{\mathbf M}_{Y}
\sim_{\mathbb{Q}}K_{S}+B_{S}-f_{S}^{*}(K_{Y}+\mathbf{B}_{Y}).$
Let $g\colon S\to T$ be the Stein factorization of $f_{S}$. 
The finite morphism $T\to Y$ is denoted by $f_{T}$. 
We put $B_{T}=g_{*}B_{S}$. 
Then the relation 
$K_{S}+B_{S}=g^{*}(K_{T}+B_{T})$ holds because 
$K_S+B_S$ is $\mathbb Q$-linearly trivial over $Y$. 
We also have the relation
$$f_{T}^{*}{\mathbf M}_{Y}\sim_{\mathbb{Q}}
K_{T}+B_{T}-f_{T}^{*}(K_{Y}+\mathbf{B}_{Y}).$$
To show that $-{\mathbf M}_{Y}$ is 
$\mathbb{Q}$-linearly equivalent to an 
effective $\mathbb{Q}$-divisor, it is sufficient to prove 
that $-\bigl(K_{T}+B_{T}-f_{T}^{*}(K_{Y}+\mathbf{B}_{Y})\bigr)$ 
is $\mathbb{Q}$-linearly equivalent to an effective $\mathbb{Q}$-divisor. 

By the definition of the 
discriminant $\mathbb{Q}$-b-divisor 
(see Definition \ref{x-def3.5}), for every prime divisor $P$ on $Y$, we have
$\coeff_{P}(\mathbf{B}_{Y})=1-b_{P}$
where $b_{P}$ is the log canonical threshold of $(X^{\nu}, B^{\nu})$ 
with respect to $\nu^{*}f^{*}P$ over the generic point of $P$. 
Since $f_{T}$ is finite, we may write 
$f_{T}^{*}P=\sum_{Q_{i}}m_{i}Q_{i}$, 
where $Q_{i}$ runs over prime divisors 
on $T$ with $f_{T}(Q_{i})=P$ and $m_{i}$ 
is the multiplicity of $Q_{i}$ with respect to $f_{T}$. 
By the ramification formula, over a 
neighborhood of the generic point of $P$ we may write
\begin{equation*}
\begin{split}
f_{T}^{*}(K_{Y}+\mathbf{B}_{Y})=&f_{T}^{*}(K_{Y}+(1-b_{P})P)\\
=&K_{T}-\sum_{Q_{i}}(m_{i}-1)Q_{i}+(1-b_{P})\sum_{Q_{i}}m_{i}Q_{i}\\
=&K_{T}+\sum_{Q_{i}}(1-m_{i}b_{P})Q_{i}.
\end{split}
\end{equation*}
We define $E:=\sum_{Q_{i}}(\coeff_{Q_{i}}(B_{T})-(1-m_{i}b_{P}))Q_{i}$. 
Then, over a neighborhood of 
the generic point of $P$, we have
$$f_{T}^{*}{\mathbf M}_{Y}\sim_{\mathbb{Q}}
K_{T}+B_{T}-f_{T}^{*}(K_{Y}+\mathbf{B}_{Y})=
\sum_{Q_{i}}(\coeff_{Q_{i}}(B_{T})-(1-m_{i}b_{P}))Q_{i}=E.$$
On the other hand, by the 
definition of $b_{P}$ 
(see Definition \ref{x-def3.5}) and the fact that $S$ 
is a log canonical center of $(X^{\nu},B^{\nu})$, the pair 
$(S, B_{S}+b_{P}f_{S}^{*}P)$ is sub log canonical 
over the generic point of $P$. 
Since $g\colon S\to T$ is birational and $K_{S}+B_{S}=g^{*}(K_{T}+B_{T})$, 
the pair 
$(T,B_{T}+b_P f_{T}^{*}P)$ is sub log canonical 
over the generic point of $P$. 
This shows $\coeff_{Q_{i}}(B_{T})+m_{i}b_{P}\leq 1$ 
for all $Q_{i}$ such that $f_{T}(Q_{i})=P$. 
Thus, $-E$ is effective. 
Hence 
$-{\mathbf M}_{Y}$ is $\mathbb{Q}$-linearly 
equivalent to an effective $\mathbb{Q}$-divisor. 

Finally, since ${\mathbf M}_{Y}$ is nef, we see that 
${\mathbf M}_{Y}\sim_{\mathbb{Q}}0$. 
\end{proof}

We prove the b-semi-ampleness of $\mathbf M$ for basic slc-trivial 
fibrations of relative dimension one under some extra assumption. 

\begin{thm}\label{x-thm7.2}
Let $f\colon (X, B)\to Y$ be a basic 
$\mathbb R$-slc-trivial fibration with $\dim X=\dim Y+1$ 
such that the horizontal part $B^h$ of $B$ is effective. 
Then the moduli $\mathbb R$-b-divisor $\mathbf M$ is b-semi-ample. 
\end{thm}

\begin{proof}
By Theorem \ref{x-thm5.6}, we may assume 
that $f\colon (X, B)\to Y$ is a basic $\mathbb Q$-slc-trivial fibration. 
By \cite[Lemma 4.12]{fujino-slc-trivial}, we 
may further assume that $Y$ is a complete variety. 
When $X$ is reducible, by the definition 
of basic slc-trivial fibrations (see Definition \ref{x-def3.2}), 
there is a stratum $S$ of $X$ such that the morphism 
$S \to Y$ is generically finite and surjective since $\dim X=\dim Y+1$. 
Thus, we can apply Theorem \ref{x-thm7.1}. 
By Theorem \ref{x-thm7.1}, the moduli $\mathbb Q$-b-divisor 
$\mathbf M$ is b-semi-ample when $X$ is reducible. 
So we may assume that $X$ is irreducible. 
Let $F$ be a general fiber of $f$. 
Then $B|_{F}\geq 0$ by the assumption $B^h\geq 0$. 
If $(F,B|_{F})$ is not kawamata log terminal, then 
there is a log canonical center $S'$ of $(X, B)$, 
that is, $S'$ is a stratum of $(X, B)$, 
such that the morphism $S' \to Y$ is generically finite and surjective. 
As in the reducible 
case, by applying Theorem \ref{x-thm7.1}, 
we see that the moduli $\mathbb Q$-b-divisor $\mathbf M$ is b-semi-ample. 
If $(F,B|_{F})$ is kawamata log terminal, 
then the morphism $f\colon (X,B) \to Y$ 
satisfies \cite[Assumption 7.11]{ps}. 
Therefore, by \cite[Theorem 8.1]{ps}, 
the moduli $\mathbb Q$-b-divisor $\mathbf M$ is b-semi-ample. 
In this way, in any case, the 
moduli $\mathbb Q$-b-divisor $\mathbf M$ is b-semi-ample. 
\end{proof}

By combining Theorem \ref{x-thm7.2} with 
the proof of Theorem \ref{x-thm1.2}, we obtain the following result, which 
generalizes Kawamata's theorem (see \cite[Theorem 1]{kawamata}).  

\begin{cor}[Adjunction and Inversion of Adjunction in codimension 
two]\label{x-cor7.3}
Under the same notation as in Theorem \ref{x-thm1.2}, 
we further assume that $\dim W=\dim X-2$. 
Then $\mathbf M$ is b-semi-ample. 
Equivalently, $M_{Z'}$ is semi-ample. 
In particular, there 
exists an effective $\mathbb{R}$-divisor $\Delta_{Z}$ on $Z$ such that
\begin{itemize}
\item
$\nu^*(K_X+\Delta)\sim _{\mathbb R}K_{Z}+\Delta_{Z}$, 
\item
$(Z,\Delta_{Z})$ is log canonical if and only if $(X,\Delta)$ is 
log canonical near $W$, and 
\item
$(Z,\Delta_{Z})$ is kawamata log terminal 
if and only if $(X, \Delta)$ is log canonical near $W$ and 
$W$ is a minimal log canonical center of $(X,\Delta)$. 
\end{itemize}
When $K_X+\Delta$ is $\mathbb Q$-Cartier, we further 
make $\Delta_Z$ an effective $\mathbb Q$-divisor on $Z$ 
such that $\nu^*(K_X+\Delta)\sim _{\mathbb Q} K_Z+\Delta_Z$ 
in the above statement. 
\end{cor}

\begin{proof} 
We use the same notation as in Theorem \ref{x-thm1.2}. 
Note that $W$ is a codimension two log canonical center of 
$(X, \Delta)$ by assumption. 
Let $f\colon Y\to X$ be a projective birational morphism 
from a smooth quasi-projective variety $Y$ such that 
$K_Y+\Delta_Y=f^*(K_X+\Delta)$ and 
that $\Supp \Delta_Y$ is a simple normal crossing divisor on $Y$. 
Without loss of generality, we may assume that 
$f^{-1}(W)$ is a simple normal crossing divisor on $Y$ 
such that $f^{-1}(W)=\sum _i E_i$ is the irreducible decomposition. 
We put 
$$
E=\sum _{a(E_i, X, \Delta)=-1}E_i. 
$$ 
We define $\Delta_E$ by $K_E+\Delta_E=(K_Y+\Delta_Y)|_E$. 
In this situation, we can check that $\Delta_E$ is effective over 
the generic point of $W$. 
Indeed, if $X$ is a surface then we can check this 
fact by using the minimal resolution. 
In the general case, by shrinking $X$ and cutting $X$ 
by general hyperplanes, we can reduce the problem to 
the case where $X$ is a surface. 

Let $Z$ be the normalization of $W$. 
By the same arguments as in Steps \ref{x-step1.3.1}, 
\ref{x-step1.3.2}, and \ref{x-step1.3.3} in the proof of 
Theorem \ref{x-thm1.2}, we can construct a 
basic $\mathbb{R}$-slc-trivial fibration $f\colon (V,\Delta_{V})\to Z$. 
Then $\dim V=\dim Z+1$ because $\dim V=\dim X-1$ and $W$ 
is a codimension two log canonical center of $(X, \Delta)$. 
Furthermore, by the discussion in the first paragraph, 
we see that the horizontal part $\Delta^h_{V}$ of $\Delta_V$ 
with respect to $f\colon V\to Z$ is effective. 
By the same arguments as in Steps \ref{x-step1.3.4}, \ref{x-step1.3.5}, 
and \ref{x-step1.3.6} in the proof of Theorem \ref{x-thm1.2}, 
we get a projective birational morphism 
$p\colon Z'\to Z$ from a smooth quasi-projective variety 
$Z'$ satisfying (i)--(v) of Theorem \ref{x-thm1.2}. 
Moreover, by Theorem \ref{x-thm7.2}, 
$\mathbf M$ is b-semi-ample, 
that is, $\mathbf M_{Z'}$ is semi-ample. 

Let $N\sim_{\mathbb{R}}\mathbf M_{Z'}$ be a general 
effective $\mathbb R$-divisor 
such that $N$ and $\mathbf B_{Z'}$ have 
no common components, $\Supp (N+ \mathbf{B}_{Z'})$ is a simple 
normal crossing divisor 
on $Z'$, and all the coefficients of $N$ are less than one. 
We put $\Delta_{Z}=p_{*}N+\mathbf B_{Z}$. 
Then, it is easy to see that $\Delta_{Z}$ satisfies the desired three 
conditions of Corollary \ref{x-cor7.3}. 
By the above construction, we can make $\Delta_Z$ an effective 
$\mathbb Q$-divisor such that $K_Z+\Delta_Z\sim _{\mathbb Q} 
\nu^*(K_X+\Delta)$ when $K_X+\Delta$ is $\mathbb Q$-Cartier. 
So we are done. 
\end{proof}

%%%%%%%%%%%%%%%

\end{document}